\documentclass[12pt,twoside,reqno]{amsart}
\usepackage[all]{xy}
\usepackage{amssymb,latexsym,amsthm,amsmath,mathtools}
\usepackage{enumitem,color}

\usepackage{tabularx}

\usepackage[colorlinks=true, linkcolor=black, citecolor=blue]{hyperref}
\usepackage[capitalize,nameinlink]{cleveref}
\usepackage{cleveref}
\usepackage[
backend=biber,
style=numeric-comp,
doi=true,
url=false,
maxbibnames=99,
giveninits=true
]{biblatex}
\crefname{section}{Section}{Sections}
\crefname{subsection}{Subsection}{Subsections}
\crefname{subsubsection}{Subsubsection}{Subsubsections}
\usepackage{parskip}
\usepackage{mathtools}

\usepackage[backend=biber, doi=true, url=false, maxbibnames=99, giveninits=true]{biblatex}
\usepackage{orcidlink}
\long\def\symbolfootnote[#1]#2{\begingroup%
	\usepackage{hyperref}
	\usepackage{tikz}
	\usepackage{blkarray}
	\usepackage{arydshln}
	\usepackage{tabularx}
	\usepackage{listings}
	\usepackage{xcolor}
	
	\lstdefinestyle{sage}{
		language=Python,
		basicstyle=\ttfamily\small,
		breaklines=true,
		frame=single,
		columns=fullflexible,
		showstringspaces=false
	}
	\setlength\dashlinedash{0.5pt}
	\setlength\dashlinegap{1.5pt}
	\setlength\arrayrulewidth{0.3pt}
	\def\thefootnote{\fnsymbol{footnote}}\footnote[#1]{#2}\endgroup}

\makeatletter
\def\imod#1{\allowbreak\mkern10mu({\operator@font mod}\,\,#1)}
\makeatother

\makeatletter
\renewcommand\subsubsection{\@startsection{subsubsection}{3}{\z@}%
	{-3.25ex\@plus -1ex \@minus -.2ex}%
	{1.5ex \@plus .2ex}%
	{\normalfont\normalsize\bfseries}}
\makeatother
\newtheorem{theorem}{Theorem}[section]
\newtheorem{lemma}[theorem]{Lemma}
\newtheorem{corollary}[theorem]{Corollary}
\newtheorem{proposition}[theorem]{Proposition}
\newtheorem*{theorem*}{Theorem}
\theoremstyle{definition}
\newtheorem{definition}[theorem]{Definition}
\newtheorem{remark}[theorem]{Remark}
\newtheorem{example}[theorem]{Example}
\newtheorem{question}[theorem]{Question}

\numberwithin{equation}{section}

\newcommand{\ignore}[1]{}

\newcommand{\mynote}[1]{}
\DeclareFieldFormat*{title}{#1}
\begin{document} 
    \author[de Mello T. C.]{Thiago Castilho de Mello\,\orcidlink{0000-0002-6988-0910}}
	\address{Universidade Federal de S\~ao Paulo, Av. Cesare Monsuetto Giulio Lattes 1201, S\~ao Jos\'e dos Campos, Brazil}
	\email{tcmello@unifesp.br}
    
    \author[Saini P.]{Prachi Saini\,\orcidlink{0009-0006-7627-1538}}
	\address{Indian Institute of Science Education and Research Mohali, Sector 81, Mohali 140306, India}
	\email{prachi2608saini@gmail.com}

	\setcounter{section}{0}
	\title{Generalized commutator on octonion algebra}
	
	\begin{abstract}
		Let $\mathcal O$ be an octonion algebra over a field $F$ of characteristic different from $2$, and let $A,B\in\mathcal O$, not both zero. We study the image of the generalized commutator $\Phi(x,y)=A(xy)-B(yx)$ on $\mathcal O$. We prove that $\Phi$ is surjective if and only if either one of $A$ and $B$ is invertible and $A\neq B$, or $A$ and $B$ are nonzero zero divisors and $A-B$ is invertible. In the remaining cases, the image of $\Phi$ is a hyperplane, a four-dimensional maximal totally isotropic subspace of $\mathcal O$, or, when $A$ and $B$ are linearly independent zero divisors and $A-B$ is not invertible, the six-dimensional subspace $A\mathcal O+B\mathcal O$. In particular, the image of $\Phi$ is always a vector subspace of $\mathcal O$. Along the way, we show that every octonion of trace zero is a commutator. We also prove that the image of $Axy-Byx$ on any quaternion algebra over $F$ is a vector subspace, extending a result of Panja S., Saini P. and Singh A. (Images
        of polynomial maps with constants, Mathematika 71 (2025), no. 3, Paper No. e70031) for $2\times 2$ matrices over algebraically closed fields, and we give examples of other generalized polynomials whose images are not vector subspaces.
	\end{abstract}
	\maketitle
	\raggedbottom
	\section{Introduction}
	\label{sec:intro}
	
	Commutators play a central role in the study of the structure of an algebra. The well-known Albert--Muckenhoupt--Shoda Theorem (see \cite{Shoda} for fields of characteristic zero and \cite{MR83961} for arbitrary fields) states that, for a field $F$, a matrix $a\in M_n(F)$ is a commutator if and only if $a$ has trace zero, i.e., $a\in\mathfrak{sl}_n(F)$. In other words, the image of the polynomial $xy-yx$ evaluated on $M_n(F)$ is the vector space $\mathfrak{sl}_n(F)$. An analogous result holds for octonion algebras over a field of characteristic different from $2$: an element of an octonion algebra $\mathcal O$ is a commutator if and only if it has trace zero (Lemma~\ref{lem:trace-zero-commutator}). This is the particular case $A=B=1$ of our main result (Theorem~\ref{thm:main}).
	
	These facts are related to problems concerning images of polynomials on algebras. Given an $F$-algebra $R$ and a polynomial $f(x_1,\dots,x_m)$ in noncommuting variables over $F$, one obtains naturally a map $R^m\longrightarrow R$ by evaluation. Polynomials that vanish under the evaluation are the polynomial identities of $R$, and this is related to the well-developed theory of polynomial identities in algebras (PI-theory). One of the main problems in the study of images of polynomials is the well-known L'vov--Kaplansky conjecture, which asks whether the image of a multilinear polynomial evaluated on the $n\times n$ matrix algebra $M_n(F)$ is always a vector space, which is equivalent to saying that such an image is one of the following: $\{0\}$, $FI_n$, $\mathfrak{sl}_n(F)$ or $M_n(F)$. For further results concerning images of polynomials, we refer the reader to the survey article \cite{KBMRY2020}.
	
	Although simple to state, this conjecture has only a few known cases of positive solutions. Kanel-Belov, Malev and Rowen solved it for $2\times2$ matrices over quadratically closed fields \cite{KBMR2012}, and Malev for $2\times2$ matrices over the real numbers \cite{Malev2014}. For $3\times3$ matrices, partial results were obtained in \cite{KBMR2016}. Analogous problems for other algebras have also been studied. For instance, the corresponding conjecture for finite-dimensional simple algebras was solved for the algebra of quaternions by Malev \cite{Malev2021}, and for the algebra of octonions, over a class of fields which includes the quadratically closed fields and $\mathbb R$, by Kanel-Belov, Malev, Pines and Rowen \cite{KBMPR2024}. In the latter case, the image of a multilinear polynomial is $\{0\}$, $F$, the space of pure octonions, or the whole algebra.
	
	Other related questions have been considered for maps defined by generalized polynomials, also called polynomials with constants, that is, polynomials whose coefficients are taken in the algebra itself. Panja, Saini and Singh \cite{PanjaSainiSingh2025} considered two families of such maps on $M_2(K)$, where $K$ is an algebraically closed field: the generalized power maps $Ax^{k_1}+By^{k_2}$, for which they characterized surjectivity, and the generalized commutator $Axy-Byx$, whose image they proved to be always a vector space. Later, Saini and Singh \cite{SAINI2027256}, characterized surjectivity of $Ax^k+By^k$ on $M_n(F)$ in terms of the nullity of the coefficient $B$ with $A$ being invertible. In \cite{Rodrigues2025}, Rodrigues considered nonzero generalized polynomials on $M_2(D)$, where $D$ is a Makar-Limanov algebraically closed division ring, and proved that every element of $M_2(D)$ is a sum of at most two elements of the image of such a polynomial.
	Later, Panja, Saini and Singh \cite{PanjaSainiSingh2026} considered the split octonion algebra over an algebraically closed field, and characterized, in terms of the orbits of the group $G_2$ acting on pairs of octonions, the pairs $(A,B)$ for which the map $A(x^{k_1})+B(y^{k_2})$ is surjective.
	
	In the present paper, we consider the generalized commutator on octonion algebras. Let $\mathcal O$ be an octonion algebra over a field $F$ with $\operatorname{char}F\neq2$, with norm $N$ and associated bilinear form $N(x,y)=N(x+y)-N(x)-N(y)$. We do not assume that $F$ is algebraically closed, nor that $\mathcal O$ is split. Our main result is the following.
	
	\begin{theorem}
		\label{thm:main}
		Let $\mathcal O$ be an octonion algebra over a field $F$ with $\operatorname{char}F\neq2$, and let $A,B\in\mathcal O$ with $A\neq0$. Consider the map
		$$
		\Phi:\mathcal O\times\mathcal O\longrightarrow\mathcal O,
		\qquad
		\Phi(x,y)=A(xy)-B(yx).
		$$
		Then $\operatorname{Im}(\Phi)$ is a vector subspace of $\mathcal O$. More precisely, $\Phi$ is surjective if and only if one of the following conditions holds:
		\begin{enumerate}[label=(\alph*)]
			\item at least one of $A$ and $B$ is invertible, and $A\neq B$;
			\item $A$ and $B$ are nonzero zero divisors and $N(A,B)\neq0$.
		\end{enumerate}
		In the remaining cases, the following statements hold:
		\begin{enumerate}[label=(\roman*)]
			\item if $A=B$ is invertible, then $\operatorname{Im}(\Phi)=A^\perp=\{z\in\mathcal O\mid N(A,z)=0\}$, which is a seven-dimensional subspace of $\mathcal{O}$;
			\item if $A$ is a nonzero zero divisor and $B\in FA$, then $\operatorname{Im}(\Phi)=A\mathcal O$, which is a four-dimensional maximal totally isotropic subspace of $\mathcal O$;
			\item if $A$ and $B$ are linearly independent nonzero zero divisors with $N(A,B)=0$, then $\operatorname{Im}(\Phi)=A\mathcal O+B\mathcal O$, which is a six-dimensional subspace of $\mathcal O$.
		\end{enumerate}
	\end{theorem}
	
	Notice that, since a division octonion algebra has no nonzero zero divisors, cases (b), (ii) and (iii) occur only when $\mathcal O$ is split. Theorem~\ref{thm:main} shows that, as in the case of $2\times 2$ matrices \cite[Theorem~B]{PanjaSainiSingh2025}, the image of the generalized commutator on an octonion algebra is always a vector space.
	
	The paper is organized as follows. In Section~\ref{sec:preliminary-cases} we treat some preliminary cases, which include the case $A=B$ and the description of the image of the left multiplication by a zero divisor. In Section~\ref{sec:invertible-case} we consider the case in which one of the coefficients is invertible. Here we use the classification of the subalgebras of an octonion algebra, due to Knarr and Stroppel \cite{MR4812310}. In Section~\ref{sec:zero-divisors} we consider the case in which both coefficients are zero divisors, and we prove Theorem~\ref{thm:main}. The main step is the case of two linearly independent zero divisors $A$ and $B$ with $N(A,B)=0$, in which, using the Peirce decomposition of $\mathcal O$ with respect to an idempotent and some explicit automorphisms of $\mathcal O$, we reduce the pair $(A,B)$ to five normal forms (Proposition~\ref{prop:independent zero divisors}). In Section~\ref{sec:quaternion} we consider the analogous problem for quaternion algebras, and we show that the image of $Axy-Byx$ on a quaternion algebra over $F$ is always a vector subspace (Theorem~\ref{thm:quaternion}). This extends \cite[Theorem~B]{PanjaSainiSingh2025} to arbitrary fields of characteristic different from $2$ and to division quaternion algebras. Finally, in Section~\ref{sec:concluding} we relate Lemma~\ref{lem:trace-zero-commutator} to known results on commutators in central simple algebras, we consider other generalized polynomials of degree $2$, and we state some questions.
	
	\section{Preliminaries}
	\label{sec:preliminaries}
	In this section, we recall some basic notions and properties of octonion algebras that will be used throughout the paper. Unless otherwise stated, $F$ denotes a field of characteristic different from $2$. The material in this section is standard and can be found in \cite[Chapter~1]{MR1763974}.
	
	Let $\mathcal{C}$ be a unital algebra over $F$.
	\begin{definition}
		A \emph{composition algebra} is a pair $(\mathcal{C},N)$, where $\mathcal{C}$ is a unital algebra over $F$ and $N:\mathcal{C}\to F$ is a nondegenerate quadratic form satisfying
		$$N(xy)=N(x)N(y)$$
		for all $x,y\in\mathcal{C}$. The quadratic form $N$ is called the \emph{norm} of $\mathcal{C}$.
	\end{definition}
	The \emph{canonical involution}, or \emph{conjugation}, of $\mathcal{C}$ is defined by
	$$
	\overline{\phantom{x}}:\mathcal{C}\longrightarrow\mathcal{C},
	\qquad
	x\longmapsto \overline{x} = N(x,1)1-x,
	$$
	where $N(\cdot,\cdot)$ denotes the bilinear form associated with the norm.
	
	For a composition algebra $\mathcal{C}$, the \emph{trace} of an element $x\in\mathcal{C}$ is defined by
	$$
	\operatorname{Tr}(x)=x+\overline{x}.
	$$
	In particular, for a quaternion algebra $Q$, viewed as a central simple algebra of degree $2$, the trace coincides with the reduced trace (see \cite{MR1286824}), that is,
	$$\operatorname{Trd}(x)=\operatorname{Tr}(x)=x+\overline{x}$$
	for all $x\in Q$.

	We will use the following standard identities (see, for example \cite{MR1763974}).
	\begin{lemma}
		\label{lem: identities}
		For all $a,b,c \in\mathcal{C}$, the following identities hold:
		\begin{enumerate}
			\item $a\overline{a}=\overline{a}a=N(a)1$.
			\item $\overline{\overline{a}}=a$.
			\item The element $a$ is invertible if and only if $N(a)\neq 0$. In this case, $a^{-1}=N(a)^{-1}\overline{a}$.
			\item $a(\overline{a}b)=N(a)b$. \label{eqn: identity 4}
			\item $ a(ab)=a^2b,\qquad (ab)b=ab^2,\qquad (ab)a=a(ba)$.
			\item $a^2-\operatorname{Tr}(a)a+N(a)1 = 0$
			\item $N(ca,b) = N(a,\overline{c}b)$ and $N(ac,b) = N(a,b\overline{c})$.
		\end{enumerate}
		The identities in $(5)$ are referred to as the \emph{alternative laws}. The identity $(6)$ is referred to as the \emph{Cayley--Hamilton identity}.
	\end{lemma}
	The following construction realizes every octonion algebra as a
	Cayley--Dickson double of a quaternion algebra; see
	\cite[Proposition~1.5.1]{MR1763974}.
	\begin{proposition}
		\label{prop: Cayley-Dickson}
		Let $(\mathcal C, N)$ be a composition algebra over a field $F$. Suppose $\mathcal C_0 \subseteq \mathcal C$ is a proper composition subalgebra and let $\omega \in \mathcal C_0^\perp$ be an element such that $N(\omega) \neq 0$.
		Then the space $$\mathcal C_1=\mathcal C_0\oplus \mathcal C_0 \omega$$ forms a composition subalgebra of $\mathcal C$ with $\dim(\mathcal C_1)=2\dim(\mathcal C_0)$. Moreover, for $a,a',b,b'\in \mathcal C_0$, the operations on $\mathcal C_1$ are given by:
		\begin{align*}
			\text{Multiplication:}\quad & (a+b\omega)(a'+b'\omega)=\big(aa'+ \delta\,\overline{b'}\, b\big)+\big( b'a+b\overline{a'}\big) \omega,\\
			\text{Norm:}\quad           & N(a+b\omega)=N(a)-\delta\,N(b),\\
			\text{Conjugation:}\quad    & \overline{a+b\omega}=\overline{a}-b\omega,
		\end{align*}
		where $\overline{\cdot}$ denotes the canonical involution on $\mathcal C_0$ and $\delta = -\,N(\omega)$.
	\end{proposition}
	Throughout the remainder of the paper, let $\mathcal{O}$ denote an octonion algebra over a field $F$ with $\operatorname{char}F\neq 2$. By the Cayley--Dickson construction, every octonion algebra admits a decomposition
	$$\mathcal{O}=\mathcal H\oplus \mathcal H\omega,$$
	where $\mathcal H$ is a quaternion subalgebra and $N(\omega)=-\delta\in F^\times$. Moreover, every quaternion subalgebra $\mathcal H$ of $\mathcal O$ appears in such a decomposition: since $N$ is nondegenerate on $\mathcal H$, it is also nondegenerate on $\mathcal H^\perp$, and it is enough to take any $\omega\in\mathcal H^\perp$ with $N(\omega)\neq0$ in Proposition~\ref{prop: Cayley-Dickson}.
	
	When the quadratic form $N$ is isotropic (i.e., when there exists $0\neq x\in \mathcal O$ such that $N(x) = 0$), $\mathcal O$ is called the split octonion algebra. It is not a division algebra and has nonzero zero divisors. We write $\mathcal O_0 = 1^\perp = \{x\in \mathcal O\mid \operatorname{Tr}(x) = 0\}$.
	
	When $\mathcal O$ is split, we realize it for the rest of the paper as the Cayley--Dickson double of the split quaternion algebra $\mathcal Q:=M_2(F)$: writing $\overline X:=\operatorname{tr}(X)I-X$ for the classical adjugate on $\mathcal Q$ (which agrees with the composition-algebra conjugate of $\mathcal Q$, viewed as a $4$-dimensional composition algebra with norm $\det$), \cref{prop: Cayley-Dickson} gives
	$$
	\mathcal O=\mathcal Q\oplus\mathcal Q\omega,\qquad \omega^2=1,
	$$
	with multiplication, norm, trace, and conjugation
	\begin{align}
		(a+x\omega)(b+y\omega)&=\bigl(ab+\overline yx\bigr)+\bigl(ya+x\overline b\bigr)\omega, \label{eq:doubling-mult}\\
		N(a+x\omega)&=\det(a)-\det(x), \qquad \operatorname{Tr}(a+x\omega)=\operatorname{tr}(a),\\
		\overline{a+x\omega}&=\overline a-x\omega,
	\end{align}
	for $a,b,x,y\in \mathcal Q$. We fix the basis $\{E_{11},E_{12},E_{21},E_{22}\}$ of matrix units for $\mathcal Q$, and write
	$$
	p:=E_{11},\qquad q:=E_{22},\qquad n:=E_{12},\qquad m:=E_{21}.
	$$
	(In \cite{MR4812310}, the elements $p$ and $n$ are denoted by $p_0$ and $n_0$.)
	Since any two split octonion algebras over $F$ are isomorphic (an octonion algebra is determined, up to isomorphism, by its norm form, and an isotropic norm form of an octonion algebra is hyperbolic, see \cite[Chapter~1]{MR1763974}), this is no loss of generality.
	
	We stress that this model fixes one particular quaternion subalgebra $\mathcal Q=M_2(F)$ of $\mathcal O$, which is split. The algebra $\mathcal O$ may also contain quaternion subalgebras which are division algebras. Indeed, if $\mathcal H$ is any quaternion algebra over $F$, then the Cayley--Dickson double of $\mathcal H$ with $\delta=1$ is an octonion algebra with norm $N(a+b\omega)=N(a)-N(b)$, which is isotropic (take $a=b=1$); hence it is the split octonion algebra, and $\mathcal H$ is isomorphic to a quaternion subalgebra of $\mathcal O$. For this reason, throughout the paper $\mathcal Q$ always denotes the fixed subalgebra $M_2(F)$ of the model above, while an arbitrary quaternion subalgebra of $\mathcal O$, split or not, is denoted by $\mathcal H$.
	
	The elements $p$ and $n$ are nonzero zero divisors of $\mathcal O$ (indeed of $\mathcal Q$), with $N(p)=N(n)=0$, $\operatorname{Tr}(p)=1$, $\operatorname{Tr}(n)=0$, and $\overline{p}=q$. We record two elementary but recurrent identities, immediate from \eqref{eq:doubling-mult}:
	\begin{equation}
		nm=p,\qquad nq=n. \label{eq:n-identities}
	\end{equation}

	\section{Some preliminary cases of the map \texorpdfstring{$\Phi$}{Phi}}
	\label{sec:preliminary-cases}

	Throughout this section, we consider the map $\Phi = \Phi_{A,B}:\mathcal O\times\mathcal O\longrightarrow\mathcal O$ defined by $\Phi(x,y)=A(xy)-B(yx)$,
	where $A, B\in\mathcal O$ and $A$ is a nonzero element. For $x\in\mathcal O$, we denote by $\psi_x=\Phi_{A,B}(x,\cdot)$ the linear map $\mathcal O\longrightarrow\mathcal O$, $y\mapsto\Phi_{A,B}(x,y)$; this notation will be used throughout the paper. We first settle some preliminary cases in which the image of $\Phi$ can be directly determined.
	
	We first consider the case when $A-B$ is invertible.
	\begin{lemma}
		\label{lem: N(A-B) nonzero}
		Let $A,B\in\mathcal O$ such that $N(A-B)\neq0$. Then the map $\Phi(x,y)=A(xy)-B(yx)$ is surjective on $\mathcal O$.
	\end{lemma}

	\begin{proof}
		Let $Z\in\mathcal O$ be arbitrary. Taking $y=1$, we have $\Phi(x,1)=Ax-Bx=(A-B)x$. Since $N(A-B)\neq0$, the element $A-B$ is invertible. Hence, taking
		$x=(A-B)^{-1}Z$ and using (4) of Lemma~\ref{lem: identities}, we obtain
		$\Phi(x,1)=Z$. Thus, every $Z\in\mathcal{O}$ lies in the image of $\Phi$, and consequently $\Phi$ is surjective.
	\end{proof}
	
	If $A\in \mathcal O$, we denote by $A\mathcal O$ the image of the left multiplication map by $A$ on $\mathcal O$, $L_A:\mathcal O \longrightarrow \mathcal O$, given by $L_A(x) = Ax$.
	
	Even if $A-B$ is not invertible, it gives us a useful bound on the image of  $\Phi$:
	\begin{lemma}
		\label{lem:sandwich}
		For all $A,B\in\mathcal O$,
		$$
		(A-B)\mathcal O\ \subseteq\ \operatorname{Im}(\Phi)\ \subseteq\ A\mathcal O+B\mathcal O.
		$$
	\end{lemma}
	\begin{proof}
		For the right inclusion, $\Phi(x,y)=A(xy)-B(yx)\in A\mathcal O+B\mathcal O$ for all $x,y$, since $xy,yx\in\mathcal O$. For the left inclusion, take $y=1$: $\Phi(x,1)=Ax-Bx=(A-B)x$, so $(A-B)\mathcal O\subseteq\operatorname{Im}(\Phi)$.
	\end{proof}

	Now we have the case in which $B = 0$.

	\begin{lemma}
		\label{lem: B is zero}
		Let $A\in\mathcal{O}$ be nonzero. Consider the map
		$\Phi:\mathcal{O}\times\mathcal{O}\longrightarrow\mathcal{O}$ defined by $\Phi(x,y)=A(xy)$.
		Then $\operatorname{Im}(\Phi)=A\mathcal{O}$.
		In particular, if $A$ is invertible, then $\Phi$ is surjective.
	\end{lemma}

	\begin{proof}
		Since $1$ is the identity element of $\mathcal{O}$, every $z\in\mathcal{O}$ can be written as $z=z1$. Hence
		$$\{xy\mid x,y\in\mathcal{O}\}=\mathcal{O}.$$
		Therefore,
		$$\operatorname{Im}(\Phi)=\{A(xy)\mid x,y\in\mathcal{O}\}=\{Az\mid z\in\mathcal{O}\}=A\mathcal{O}.$$
		If $A$ is invertible, then left multiplication by $A$ is a bijection of $\mathcal{O}$. Thus $A\mathcal{O}=\mathcal{O}$,
		and consequently $\Phi$ is surjective.
	\end{proof}
	
	Based on the last result, a natural question is: which subspaces of $\mathcal O$ can be realized as $A\mathcal O$, for some $A\in \mathcal O$? If $A$ is zero or invertible, the answer is clear. The problem is when $A$ is a nonzero zero divisor of $\mathcal O$, i.e., $N(A) = 0$. In this case $\mathcal O$ is the split octonion algebra, and the answer is known: van der Blij and Springer proved that the maximal totally isotropic subspaces of $\mathcal O$ are exactly the subspaces of the form $A\mathcal O$ and $\mathcal O A$, with $A$ a nonzero zero divisor \cite[Theorem~3]{MR123622} (see also \cite[Theorem~3.6]{MR4812310}), and Gow proved that $\dim A\mathcal O = 4$ \cite[Lemma~1]{Gow2008}. In the next result we collect these facts, together with some consequences that will be used later.

	\begin{proposition}\label{Image_of_L_A}
		Let $A\in \mathcal O$ be a nonzero zero divisor, i.e., $A\neq 0$ and $N(A) = 0$. Then
		\begin{enumerate}
			\item $\operatorname{Im}(L_A)$ is a totally isotropic subspace of $\mathcal O$, i.e., $N(x,y) = 0$ for any $x,y\in \operatorname{Im}(L_A)$.
			\item $L_A$ satisfies $L_A^2 = \operatorname{Tr}(A)L_A$.
			\item $\dim \operatorname{Im}(L_A) = 4$.
			\item If $\operatorname{Tr}(A) = 0$, then $\operatorname{Im}(L_A) = \ker(L_A)$.
			\item If $\operatorname{Tr}(A) \neq 0$, then $\mathcal O = \operatorname{Im}(L_A)\oplus \ker(L_A)$, with both subspaces being maximal totally isotropic.
		\end{enumerate}
	\end{proposition}
	
	\begin{proof}
		Items (1) and (3) are known: by \cite[Theorem~3]{MR123622} (see also \cite[Theorem~3.6(b)]{MR4812310}), $\operatorname{Im}(L_A) = A\mathcal O$ is a maximal totally isotropic subspace of $\mathcal O$, and $\dim \operatorname{Im}(L_A) = 4$ by \cite[Lemma~1]{Gow2008}. Item (1) also follows directly from the multiplicativity of the norm: for $x\in \mathcal O$, $N(Ax) = N(A)N(x) = 0$, and since $N(x,y) = N(x+y)-N(x)-N(y)$, we obtain that $N(x,y) = 0$ for any $x,y\in \operatorname{Im}(L_A)$.
		
		(2) Since $N(A) = 0$, the Cayley--Hamilton identity gives $A^2 = \operatorname{Tr}(A)A$. Multiplying by $x$ and using the alternative laws, we obtain $A(Ax) = \operatorname{Tr}(A)(Ax)$, for any $x\in \mathcal O$, which means $L_A^2 = \operatorname{Tr}(A) L_A$.
		
		(4) If $\operatorname{Tr}(A) = 0$, $L_A^2 = 0$, then $\operatorname{Im}(L_A)\subseteq \ker(L_A)$. By (3), $\dim \operatorname{Im}(L_A) = 4$ and  by rank--nullity, $8 = \dim \mathcal O = \dim \operatorname{Im}(L_A) + \dim \ker(L_A)$, and we obtain $\dim \ker(L_A) = 4$ and this implies equality.

		(5) Now we assume $\operatorname{Tr}(A) \neq 0$. Then $f = A/\operatorname{Tr}(A)$ satisfies $\operatorname{Tr}(f) = 1$. Also, $f$ is an idempotent of $\mathcal O$, that is $f^2 = f$ and using the alternative laws, $L_f^2 = L_f$. As a consequence, $\mathcal O = \operatorname{Im}(L_f)\oplus \ker(L_f) = \operatorname{Im}(L_A)\oplus \ker(L_A)$. Notice that $\overline{f} = \operatorname{Tr}(f)1-f = 1-f$ and $N(1-f) = N(1) + N(f) - N(1,f)$. Now it follows from the facts that $N(f) =0$, $N(1) = 1$ and $N(1,f) = \operatorname{Tr}(f) = 1$ that $N(1-f) = 0$, that is, $1-f$ is a nonzero zero divisor. By item (4) of Lemma \ref{lem: identities}, $f((1-f)y) = f(\overline{f}y) = N(f)y = 0$ for all $y\in \mathcal O$, so $\operatorname{Im}(L_{1-f})\subseteq \ker(L_f)$. By (3) and rank--nullity, both subspaces have dimension $4$, hence $\ker(L_A) = \ker(L_f) = \operatorname{Im}(L_{1-f})$ (see also \cite[Theorem~3.6(a)]{MR4812310}). Finally, by \cite[Theorem~3]{MR123622}, both $\operatorname{Im}(L_A)$ and $\ker(L_A) = \operatorname{Im}(L_{1-f})$ are maximal totally isotropic subspaces of $\mathcal O$.
	\end{proof}
	
	We recall the following result from \cite[Corollary 0.9]{MR1286824}.
	\begin{lemma}\label{lem:amitsurrowencommutator}
		Let $R$ be a central simple algebra of degree $2$ over a field $F$ with $\operatorname{char}F\neq2$. Then every element of $R$ having reduced trace zero is a commutator.
	\end{lemma}
	In particular, since every quaternion algebra over $F$ is a central simple algebra of degree $2$, we have the following.
	\begin{corollary}
		\label{cor: trace zero quaternion}
		Let $Q$ be a quaternion algebra over a field $F$ with $\operatorname{char}F\neq2$. Then every element $x\in Q$ satisfying $\operatorname{Trd}(x)=0$ is a commutator in $Q$.
	\end{corollary}
	We recall the following standard proposition; see \cite[Chapter 1]{MR1763974} or \cite[1.4 (10)]{MR4812310}.
	\begin{proposition}
		\label{prop: element octonion to quaternion}
		Every element $a$ of an octonion algebra $\mathcal O$ over $F$ is contained in a quaternion subalgebra of $\mathcal O$. If $\mathcal O$ is split, this quaternion subalgebra may be taken to be split as well.
	\end{proposition}
	
	\begin{lemma}\label{lem:trace-zero-commutator}
		If $\mathcal O$ is an octonion algebra, then the image of $\Phi(x,y) = xy-yx$ is $\mathcal O_0$, the set of pure octonions.
	\end{lemma}
	
	\begin{proof}
		First notice that \[\operatorname{Tr}(xy) = N(xy,1) = N(x,\overline{y}) = N(1,\overline{x}\overline{y}) = N(1,\overline{yx}) = N(yx,1) = \operatorname{Tr}(yx).\] As a consequence, $\operatorname{Tr}(xy-yx) = 0$, and we get the first inclusion.
		
		On the other hand, let $z\in \mathcal O$ be a trace zero element. By Proposition \ref{prop: element octonion to quaternion}, $z$ lies in a quaternion subalgebra $\mathcal H$ and $0 = \operatorname{Tr}_{\mathcal O}(z) = \operatorname{Tr}_{\mathcal H}(z)$. So $z$ is a trace zero element in a quaternion algebra. By Corollary \ref{cor: trace zero quaternion} $z$ is a commutator.
	\end{proof}
	
	We notice that the above result is a continuation of a result started by Shoda in \cite{Shoda}, later improved by Albert and Muckenhoupt in \cite{MR83961}, see Section \ref{sec:concluding} for the precise statement and its place in that literature.
	
	\begin{lemma}
		\label{lem: A=B image}
		Let $A\in\mathcal O\setminus\{0\}$. Then the image of the map
		$\Phi(x,y)=A(xy-yx)$ is a proper vector subspace of $\mathcal O$ of dimension $\leq 7$.
	\end{lemma}
	
	\begin{proof}
		By the previous Lemma, the image of $xy-yx$ is $\mathcal O_0$, a $7$-dimensional vector subspace of $\mathcal O$. Hence, the image of $\Phi$ is the image of the linear map $L_A$ on $\mathcal O_0$, also a vector space, denoted by $A\mathcal O_0$ of dimension $\leq 7$.
	\end{proof}
	
	As before, it is natural to ask which subspaces can be realized as $A\mathcal O_0$. We will show that if $A$ is invertible, then the image is $A^{\perp} = \{x\in \mathcal O\mid  N(A,x) = 0\}$ and if $A$ is a nonzero zero divisor, then $A\mathcal O_0 = A\mathcal O$, which is a totally isotropic $4$-dimensional subspace of $\mathcal O$.
	
	Let us start with the invertible case:
	
	\begin{proposition}
		\label{prop: A=B invertible case}
		Let $A \in \mathcal O$ be invertible. Then the image of $\Phi(x,y) = A(xy-yx)$ evaluated on $\mathcal O$ is \[A^\perp = \{x\in \mathcal O\mid  N(A,x) = 0\}.\] This is a $7$-dimensional subspace (hyperplane) of $\mathcal O$. Moreover, $A\mathcal O_0 = \mathcal O_0$ if and only if $A = \alpha 1$, for some $\alpha \in F$.
	\end{proposition}
	
	\begin{proof}
		First, notice that for any $x,y\in \mathcal O$, we have \[N(Ax,Ay) = N(A)N(x,y).\] Indeed, it follows by the linearization of the identity $N(Ax) = N(A)N(x)$.
		
		Applying the above for $y = 1$, we have
		\[N(Ax,A) = N(A)N(x,1) = N(A)\operatorname{Tr}(x)\]
		Now if $x\in \mathcal O_0$, we obtain $N(Ax,A) = 0$, i.e., $Ax\in A^{\perp}$, which means $A\mathcal O_0\subseteq A^{\perp}$.
		
		Assuming now $A$ invertible, $L_A$ is injective, and so is its restriction to $\mathcal O_0$. Since $\mathcal O_0$ has dimension $7$, its image has also dimension $7$. Since $A^{\perp}$ is also $7$-dimensional (since $N$ is nondegenerate), the equality follows.
		
		Finally, since $\operatorname{Tr}(x) = N(x,1)$, we have $\mathcal O_0 = 1^{\perp} = (\alpha 1)^{\perp}$  for any $\alpha \in F\setminus\{0\}$. Also, since $N$ is nondegenerate, $A^\perp = 1^\perp$ implies $A = \alpha 1$, and the last statement is proved.
	\end{proof}

	\begin{proposition}
		\label{prop: zero-divisor commutator}
		Let $A\in \mathcal O$ be a nonzero zero divisor. Then $A\mathcal O_0 = A\mathcal O = \operatorname{Im}(L_A)$. In particular, $\dim A\mathcal O_0 = 4$.
	\end{proposition}
	
	\begin{proof}
		First, notice that $\mathcal O = \mathcal O_0 \oplus F1$. Hence $\operatorname{Im}(L_A) = A\mathcal O_0 + FA$. As a consequence, since $\dim \operatorname{Im}(L_A) = 4$ (by Proposition \ref{Image_of_L_A}), $\dim A\mathcal O_0 \geq 4 - 1 = 3$. On the other hand, $\dim A\mathcal O_0 = 4$ if and only if $A\in A\mathcal O_0$, if and only if there exists $x\in \mathcal O_0$ such that $Ax = A$, which is equivalent to $A(x-1) = 0$. If $k = x-1$, the latter is equivalent to finding an element $k\in \ker(L_A)$ such that $k+1$ has trace zero. But $0 = \operatorname{Tr}(k+1) = \operatorname{Tr}(k)+\operatorname{Tr}(1) = \operatorname{Tr}(k)+2$ (recall that $\operatorname{Tr}(1) = 2$). So, $\dim A \mathcal O_0 = 4$ if and only if there exists an element $k\in \ker(L_A)$ with $\operatorname{Tr}(k) = -2$.     
		Now, for $y\in \mathcal O$, let $k = \overline{A}y$. By item (4) of Lemma \ref{lem: identities}, $Ak = A(\overline{A}y) = N(A)y = 0$, that is, $k\in \ker(L_A)$. Moreover, by item (7) of Lemma \ref{lem: identities}, $\operatorname{Tr}(k) = N(\overline{A}y,1) = N(y,A)$. Since $A\neq 0$ and $N$ is nondegenerate, we may choose $y$ such that $N(y,A) = -2$. Then $k = \overline{A}y$ is an element of $\ker(L_A)$ with trace equal to $-2$, and this finishes the proof.
	\end{proof}
	
	\section{At least one coefficient is invertible}
	\label{sec:invertible-case}
	
	In this section we assume that at least one of $A$ and $B$ is invertible. Of course we may assume $A$ to be invertible (since the image of $\Phi_{A,B}$ and $\Phi_{B,A}$ is the same). 
	Also, we may assume that $A$ and $B$ are linearly independent. Indeed, the case $A = B$ was considered in Proposition \ref{prop: A=B invertible case} and if $B = \alpha A$, for some $\alpha\in F\setminus\{1\}$ we have, $\Phi(x,1) =(1-\alpha)Ax$ and the image of $\Phi$ is $A\mathcal O = \mathcal O$. So in this section we assume $A$ and $B$ linearly independent and $A$ invertible.
	
	If $\mathcal O$ is a division octonion algebra, the case is immediately settled, since in this case $A-B$ is invertible and $\Phi$ is surjective by Lemma  \ref{lem: N(A-B) nonzero}.
	
	Hence from now on in this section, we assume $\mathcal O$ is a split octonion algebra, realized as in Section~\ref{sec:preliminaries} (i.e., as $\mathcal O  = \mathcal Q \oplus \mathcal Q \omega$, where $\mathcal Q = M_2(F)$ and $\omega^2=1$). We notice, however, that $\mathcal O$ may also contain quaternion subalgebras $\mathcal H$ which are division algebras, and the subalgebra generated by $A$ and $B$ may be one of them (see Proposition~\ref{prop: T contained in quaternion} and Remark~\ref{rem:unified-coverage}).
	
	We now notice that the surjectivity of the map under consideration is invariant under the action of the group $\operatorname{Aut}_F(\mathcal O)$.
	\begin{lemma}
		\label{lem: automorphism invariance}
		Let $A,B\in\mathcal O$ and let $\sigma\in\operatorname{Aut}_F(\mathcal O)$. Then
		$$
		\operatorname{Im}(\Phi_{\sigma(A),\sigma(B)})
		=
		\sigma\bigl(\operatorname{Im}(\Phi_{A,B})\bigr).
		$$
		In particular, $\Phi_{A,B}$ is surjective if and only if
		$\Phi_{\sigma(A),\sigma(B)}$ is surjective.
	\end{lemma}
	\begin{proof}
		For all $x,y\in\mathcal O$,
		$$
		\begin{aligned}
			\sigma\bigl(\Phi_{A,B}(x,y)\bigr)
			&=\sigma\bigl(A(xy)-B(yx)\bigr)\\
			&=\sigma(A)\bigl(\sigma(x)\sigma(y)\bigr)
			-\sigma(B)\bigl(\sigma(y)\sigma(x)\bigr)\\
			&=\Phi_{\sigma(A),\sigma(B)}(\sigma(x),\sigma(y)).
		\end{aligned}
		$$
		Since $\sigma$ is bijective, it follows that
		$$
		\operatorname{Im}(\Phi_{\sigma(A),\sigma(B)})
		=
		\sigma\bigl(\operatorname{Im}(\Phi_{A,B})\bigr).
		$$
		The last assertion follows immediately, since $\sigma$ is bijective.
	\end{proof}
	
	Let $T=\langle A,B\rangle$ be the subalgebra of $\mathcal O$ generated
	by $A$ and $B$. For every
	$\sigma\in\operatorname{Aut}_F(\mathcal O)$,
	$$
	\sigma(T)=\langle\sigma(A),\sigma(B)\rangle.
	$$
	
	Hence, by Lemma~\ref{lem: automorphism invariance}, when studying
	surjectivity of $\Phi$, we may replace $T$ by any representative of its
	$\operatorname{Aut}_F(\mathcal O)$-orbit and replace $(A,B)$ by the
	corresponding pair $(\sigma(A),\sigma(B))$.
	
	We first recall two standard facts concerning the subalgebra generated
	by two elements of an octonion algebra.
	
	\begin{lemma}\cite[Lemma 2.6]{MR4812310}
		\label{lem:two generated subalgebra}
		Let $T$ be a subalgebra of an octonion algebra $\mathcal O$ over a field $F$. If $T$ is generated by two elements $x,y\in\mathcal O$, then
		$\dim_F T\leq4$. In fact,
		$$T\subseteq F1+Fx+Fy+F(xy).$$
	\end{lemma}
	
	Since $A$ is invertible, the Cayley--Hamilton identity implies $1 = N(A)^{-1}(\operatorname{Tr}(A)1-A)A$ and $T$ must be a unital algebra. By Artin's theorem and by the above Lemma, $T$ must be an associative subalgebra of $\mathcal O$ of dimension at most $4$.
	Up to the action of the automorphism group $\operatorname{Aut}_F(\mathcal O)$, such subalgebras were classified in \cite[Section~6.4]{MR4812310}.
	Recall that $\mathcal O=\mathcal Q\oplus\mathcal Q\omega$, where $\mathcal Q=M_2(F)$ and $\omega^2=1$.

	\begin{lemma}
		\label{lem: unitary subalgebras}
		Let $T\subseteq \mathcal O$ be a unital $F$-subalgebra and $\dim_F T\leq4$. Then, up to the action of $\operatorname{Aut}_F(\mathcal O)$, one of the following holds:
		\begin{enumerate}
			\item If $\dim_F T=1$, then $T=F$.
			\item If $\dim_F T=2$, then $T$ is one of the following:
			\begin{enumerate}
				\item a separable quadratic field extension of $F$;
				\item a split composition algebra, isomorphic to $F[X]/(X^2-X)\cong F\times F$;
				\item a local algebra, isomorphic to $F[X]/(X^2)$.
			\end{enumerate}
			
			\item If $\dim_F T=3$, then the radical of the restriction of the norm form to $T$ is nontrivial. Moreover, $T$ belongs to one of the following two types:
			\begin{enumerate}
				\item $F+Fp+Fn$;
				\item $F+(Fp+Fn)\omega$.
			\end{enumerate}
			
			\item If $\dim_F T=4$, then one of the following occurs:
			\begin{enumerate}
				\item $T$ is a quaternion algebra, either split or non-split;
				\item $T$ belongs to the $\operatorname{Aut}_F(\mathcal O)$-orbit of $\mathcal C+(Fp+Fn)\omega$, where $\mathcal C\subseteq\mathcal Q$ is a two-dimensional composition algebra;
				\item the radical $J=\{x\in T\mid N(x)=0\}$ of the restriction of the norm form to $T$ has dimension $3$. In this case, $J$ is isomorphic to the Heisenberg algebra, and $T$ belongs to the $\operatorname{Aut}_F(\mathcal O)$-orbit of $F+(n\mathcal O\cap\mathcal O n)$.
			\end{enumerate}
		\end{enumerate}
	\end{lemma}

	We now consider the different possibilities for $T$ arising from Lemma~\ref{lem: unitary subalgebras}. Several of these cases can be treated simultaneously. We first consider the case when $T$ is contained in a quaternion subalgebra of $\mathcal O$; see Proposition~\ref{prop: T contained in quaternion}. We then consider subalgebras of the form
	$$\mathcal C+(Fp+Fn)\omega,$$
	where $\mathcal C$ is a two-dimensional composition algebra, together with the case
	$$F+(Fp+Fn)\omega;$$ both are treated simultaneously in Proposition~\ref{prop:unified}. Finally, we consider the case
	$$F+(n\mathcal O\cap\mathcal O n),$$
	where $n\mathcal O\cap\mathcal O n$ is a three-dimensional
	Heisenberg algebra; see Proposition~\ref{prop:T Heisenberg case}.

	\subsection{Subalgebras contained in a quaternion subalgebra} 
	We first establish an auxiliary result that will be used to treat the case where the subalgebra generated by $A$ and $B$ lies within a quaternion subalgebra of $\mathcal O$.
	\begin{lemma}
		\label{lem:quaternion nonvanishing}
		Let $Q$ be a quaternion algebra over a field $F$ with
		$\operatorname{char}F\neq2$. Let $A\in Q^\times$ and $B\in Q$ with
		$B\neq A$. Put $C=A^{-1}B$ and $D=BA^{-1}$.
		For $X\in Q$, define
		$$
		T_X^C:Q\longrightarrow Q,
		\qquad
		T_X^C(Y)=XY-CYX,
		$$
		and
		$$
		f_D(X)=N\bigl(X-\overline{X}D\bigr),
		\qquad
		g_C(X)=\det(T_X^C),
		$$
		where $N$ denotes the reduced norm on $Q$.
		Then there exists $X_0\in Q$ such that
		$$
		f_D(X_0)\neq0
		\qquad\text{and}\qquad
		g_C(X_0)\neq0.
		$$
	\end{lemma}
	
	The proof will be given as  a consequence of the following auxiliary lemmas. We start with the following, which is a straightforward consequence of the fact that traces and determinants are invariant under conjugation.
	
	\begin{lemma}[Similarity invariance of $f_D$]
		\label{lem: f invariance}
		Let $D \in M_2(F)$ and define $f_D: M_2(F) \to F$ by
		$$
		f_D(X) = \det(X - \overline{X}D).
		$$
		For any $P \in GL_2(F)$, if $D' = PDP^{-1}$ and $X' = PXP^{-1}$, then
		$$
		f_{D'}(X') = f_D(X).
		$$
		In particular, $f_D \equiv 0$ if and only if $f_{D'} \equiv 0$.
	\end{lemma}
	
	\begin{lemma}[Similarity invariance of $g_C$]
		\label{lem: g invariance}
		Let $C \in M_2(F)$ and define the linear map $T_X^C: M_2(F) \to M_2(F)$ by
		$$
		T_X^C(Y) = XY - CYX.
		$$
		Define $g_C: M_2(F) \to F$ by $g_C(X) = \det(T_X^C)$. For any $P \in GL_2(F)$, if $C' = PCP^{-1}$ and $X' = PXP^{-1}$, then
		$$
		g_{C'}(X') \neq 0 \iff g_C(X) \neq 0.
		$$
		In particular, $g_C \equiv 0$ if and only if $g_{C'} \equiv 0$.
	\end{lemma}
	\begin{proof}
		Let $Y' = PYP^{-1}$. As in the previous lemma, we rely on the fact that determinants and traces are invariant under conjugation to prove that \[T_{X'}^{C'}(Y') 
		= P(T_X^C(Y))P^{-1}.\] 
	\end{proof}

	\begin{lemma}
		\label{lem:split quaternion nonvanishing}
		Let $F$ be a field with $\operatorname{char}F\neq2$. Let
		$A\in GL_2(F)$ and $B\in M_2(F)$ with $B\neq A$. Put
		$$
		C=A^{-1}B
		\qquad\text{and}\qquad
		D=BA^{-1}.
		$$
		Then there exists $X_0\in M_2(F)$ such that
		$$
		\det\bigl(X_0-\overline{X_0}D\bigr)\neq0 \qquad \text{and} \qquad
		\det\bigl(T_{X_0}^C\bigr)\neq0,
		$$
		where
		$T_X^C(Y)=XY-CYX.$
	\end{lemma}
	
	\begin{proof}
		Since $B\neq A$, we have $C\neq I$. Moreover,
		$$
		D=ACA^{-1},
		$$
		so $C$ and $D$ are similar. By
		Lemmas~\ref{lem: f invariance} and~\ref{lem: g invariance}, we may
		replace $C$ and $D$ simultaneously by conjugates. We distinguish
		three cases.
		
		\textbf{Case (1).} Suppose that $1$ is not an eigenvalue of $C$.
		
		Take $X_0=I$. Then
		$$
		T_I^C(Y)=(I-C)Y.
		$$
		Since $I-C$ is invertible,
		$$
		\det(T_I^C)\neq0.
		$$
		Since $D$ is similar to $C$, $1$ is not an eigenvalue of $D$ either.
		Therefore
		$$
		\det(I-D)\neq0.
		$$
		Thus $X_0=I$ satisfies both required conditions.
		
		\textbf{Case (2).} Suppose that $1$ is the only eigenvalue of $C$.
		
		Since $C\neq I$, after conjugation we may assume
		$$
		C=
		\begin{pmatrix}
			1&1\\
			0&1
		\end{pmatrix}.
		$$
		Take
		$$
		X_0=
		\begin{pmatrix}
			0&1\\
			1&0
		\end{pmatrix}.
		$$
		A direct calculation gives
		$$
		\det(T_{X_0}^C)=-1\neq0.
		$$
		Since $D$ is similar to $C$,
		$$
		\operatorname{tr}(D)=2,
		\qquad
		\det(D)=1.
		$$
		Moreover,
		$$
		\overline{X_0}=-X_0.
		$$
		Hence
		$$
		\begin{aligned}
			\det\bigl(X_0-\overline{X_0}D\bigr)
			&=\det\bigl(X_0(I+D)\bigr)\\
			&=\det(X_0)\det(I+D)\\
			&=-\bigl(1+\operatorname{tr}(D)+\det(D)\bigr)\\
			&=-4\neq0,
		\end{aligned}
		$$
		since $\operatorname{char}F\neq2$.
		
		\textbf{Case (3).} Suppose that $1$ is an eigenvalue of $C$ and
		the other eigenvalue is $\lambda\neq1$.
		
		Since the eigenvalues are distinct, after conjugation we may assume
		$$
		C=
		\begin{pmatrix}
			1&0\\
			0&\lambda
		\end{pmatrix}.
		$$
		Suppose first that $\lambda\neq-1$. Take
		$$
		X_0=
		\begin{pmatrix}
			0&1\\
			1&0
		\end{pmatrix}.
		$$
		A direct calculation gives
		$$
		\det(T_{X_0}^C)=(\lambda-1)^2\neq0.
		$$
		Since $D$ is similar to $C$,
		$$
		\operatorname{tr}(D)=1+\lambda,
		\qquad
		\det(D)=\lambda.
		$$
		Therefore,
		$$
		\begin{aligned}
			\det\bigl(X_0-\overline{X_0}D\bigr)
			&=-\det(I+D)\\
			&=-\bigl(1+\operatorname{tr}(D)+\det(D)\bigr)\\
			&=-2(1+\lambda)\neq0.
		\end{aligned}
		$$
		
		It remains to consider $\lambda=-1$. Thus we may assume
		$$
		C=
		\begin{pmatrix}
			1&0\\
			0&-1
		\end{pmatrix}.
		$$
		Since $D$ is similar to $C$, we have
		$$
		\operatorname{tr}(D)=0
		\qquad\text{and}\qquad
		\det(D)=-1.
		$$
		Write
		$$
		D=
		\begin{pmatrix}
			d&e\\
			h&-d
		\end{pmatrix}.
		$$
		\begin{equation}
			d^2+eh=1.
			\label{eq:detD}
		\end{equation}
		
		For $u,t\in F$, set
		$$
		X_{u,t}=
		\begin{pmatrix}
			0&1\\
			u&t
		\end{pmatrix}.
		$$
		A direct calculation gives
		$$
		\det(T_{X_{u,t}}^C)=4u^2
		$$
		and
		$$
		\det\bigl(X_{u,t}-\overline{X_{u,t}}D\bigr)
		=
		t(eu+h-dt).
		$$
		
		We choose $u,t\in F^\times$ such that
		$$
		eu+h-dt\neq0.
		$$
		If $d\neq0$, choose any $u\in F^\times$. For this fixed $u$, there
		is at most one value of $t$ satisfying
		$$
		eu+h-dt=0.
		$$
		Since $\operatorname{char}F\neq2$, the set $F^\times$ contains at
		least two elements, and hence we may choose $t\in F^\times$ different
		from this exceptional value.
		
		If $d=0$, then \eqref{eq:detD} gives $eh=1$, so $e,h\neq0$. The equation
		$$
		eu+h=0
		$$
		has the unique solution $u=-h/e$. We may therefore choose
		$$
		u\in F^\times,
		\qquad
		u\neq-\frac{h}{e},
		$$
		and take any $t\in F^\times$.
		
		In either case,
		$$
		\det\bigl(X_{u,t}-\overline{X_{u,t}}D\bigr)\neq0
		$$
		and
		$$
		\det(T_{X_{u,t}}^C)=4u^2\neq0.
		$$
		This completes the proof.
	\end{proof}

	The next is a well-known fact.

	\begin{lemma}
		\label{lem: Polynomial non-vanishing over infinite fields}
		Let $F$ be an infinite field. If $0 \neq h \in F[x_1, \ldots, x_n]$, then there exists $(a_1, \ldots, a_n) \in F^n$ such that $h(a_1, \ldots, a_n) \neq 0$.
	\end{lemma}

	We are now ready to prove Lemma~\ref{lem:quaternion nonvanishing}.
	\begin{proof}[Proof of Lemma~\ref{lem:quaternion nonvanishing}]
		We distinguish according to whether $Q$ is split.
		
		Suppose first that $Q$ is split. Then $Q\cong M_2(F),$
		and under this isomorphism the canonical involution corresponds to
		the adjugate and the reduced norm corresponds to the determinant.
		Hence the conclusion follows directly from
		Lemma~\ref{lem:split quaternion nonvanishing}.
		
		Suppose now that $Q$ is non-split. Since every quaternion algebra over
		a finite field is split, the field $F$ is infinite.
		
		Let $K/F$ be a splitting field for $Q$. Then
		$$
		Q_K:=Q\otimes_FK\cong M_2(K).
		$$
		Under such an isomorphism, the scalar extensions of $f_D$ and $g_C$
		are the corresponding polynomials on $M_2(K)$. Since
		$$
		A^{-1}B\neq1,
		$$
		Lemma~\ref{lem:split quaternion nonvanishing}, applied over $K$,
		gives an element $X_K\in Q_K$ such that
		$$
		(f_D)_K(X_K)\neq0
		\qquad\text{and}\qquad
		(g_C)_K(X_K)\neq0.
		$$
		Therefore
		$$
		(f_D)_K(g_C)_K
		$$
		is a nonzero polynomial on $Q_K$.
		
		Choose an $F$-basis of $Q$, identifying $Q$ with $F^4$. Then
		$$
		h:=f_Dg_C\in F[x_1,x_2,x_3,x_4],
		$$
		and its scalar extension to $K$ is
		$$
		h_K=(f_D)_K(g_C)_K.
		$$
		Since the natural map
		$$
		F[x_1,x_2,x_3,x_4]
		\longrightarrow
		K[x_1,x_2,x_3,x_4]
		$$
		is injective and $h_K\neq0$, it follows that
		$$
		h\neq0.
		$$
		Since $F$ is infinite, Lemma~\ref{lem: Polynomial non-vanishing over infinite fields}
		gives $X_0\in Q$ such that
		$$
		h(X_0)\neq0.
		$$
		Hence
		$$
		f_D(X_0)\neq0
		\qquad\text{and}\qquad
		g_C(X_0)\neq0.
		$$
		This completes the proof.
	\end{proof}
	
	We first consider the case when the subalgebra $T$ generated by $A$
	and $B$ is contained in a quaternion subalgebra of $\mathcal O$.
	
	\begin{proposition}
		\label{prop: T contained in quaternion}
		Let $A,B\in\mathcal O\setminus\{0\}$, where $A$ is invertible and
		$A\neq B$. Let $\Phi = \Phi_{A,B}:\mathcal O \times \mathcal O \longrightarrow \mathcal O$ given by $\Phi(x,y) = A(xy) - B(yx)$. Let $T$ be the subalgebra of $\mathcal O$ generated by
		$A$ and $B$ and suppose that $T$ is contained in a quaternion subalgebra $\mathcal H$ of $\mathcal O$. Then there exists $x_0\in \mathcal H$ such that $\psi_{x_0}$ is bijective. In particular, $\Phi$ is surjective.
	\end{proposition}
	
	\begin{proof}
		By the Cayley--Dickson process, there exists $\omega\in \mathcal H^\perp$ (not necessarily the element $\omega$ of the model of Section~\ref{sec:preliminaries}) such that $\mathcal O=\mathcal H\oplus \mathcal H\omega$ and $N(\omega)=-\delta\in F^\times$.
		Every element $y$ of $\mathcal O$ has a unique decomposition $y=y_0+y_1\omega$, with $y_i\in \mathcal H$. By Lemma \ref{lem:quaternion nonvanishing},  there exists $x_0\in \mathcal H$ satisfying $g_C(x_0)\neq 0$ and $f_D(x_0) \neq 0$. Here $C = A^{-1}B$ and $D = BA^{-1}$. For such $x_0$ and $y$ as above, we have
		$$x_0y=x_0y_0 + (y_1x_0)\omega \qquad \text{ and } \qquad  yx_0=y_0x_0 + (y_1\overline{x_0})\omega$$
		Hence, using the associativity of $\mathcal H$, we have 
		\begin{align*}
			\Phi(x_0,y) &= A(x_0y_0) - B(y_0x_0) + (y_1(x_0A)-y_1(\overline{x_0}B))\omega\\
			& = A(x_0y_0-Cy_0x_0) + (y_1(x_0-\overline{x_0}D)A)\omega
		\end{align*}
		Now, let $z = z_0+z_1\omega \in \mathcal O$, with $z_i\in \mathcal H$. Since $A$ is invertible and $g_C(x_0)\neq 0$, there exists $y_0\in \mathcal H$ such that $A(x_0y_0-Cy_0x_0) = z_0$ and since $f_D(x_0)\neq 0$, there exists $y_1\in \mathcal H$ such that $(y_1(x_0-\overline{x_0}D)A) = z_1$. Hence $z = \Phi(x_0,y)$ and the proposition is proved.
	\end{proof}

	\subsection{Subalgebras not contained in a quaternion subalgebra}
	
	We split this subsection into two parts, according to whether $T$ is contained in the subalgebra $\mathcal Q+S\omega$, where $S=Fp+Fn$, or $T\subseteq F+R$, where $R=n\mathcal O\cap\mathcal O n$ (which is a three-dimensional Heisenberg algebra, see \cite[Example~3.7(b)]{MR4812310}).
	
	\subsubsection{Coefficients in $\mathcal Q + S\omega$}
	
	Notice that for any $u,v \in S$, we have $\overline{u}v = 0$.
	
	\begin{proposition}
		\label{prop:unified}
		Let $A,B\in\mathcal O\setminus\{0\}$ with $A$ invertible and $A\neq B$. Suppose $A=a+r\omega$, $B=b+s\omega$ with $a,b\in \mathcal Q$ and $r,s\in S$. Then there is $x_0 \in \mathcal Q$ such that the linear map $\psi_{x_0}(y)  =  \Phi_{A,B}(x_0, y)$ is bijective. In particular $\Phi_{A,B}$ is surjective.
	\end{proposition}
	
	\begin{remark}
		\label{rem:unified-coverage}
		Let us explain how Propositions~\ref{prop: T contained in quaternion} and~\ref{prop:unified} together cover all the cases of Lemma~\ref{lem: unitary subalgebras}, except the Heisenberg-radical case (4)(c), which is treated in Proposition~\ref{prop:T Heisenberg case} below. By Lemma~\ref{lem: automorphism invariance}, we may replace $T=\langle A,B\rangle$ by any representative of its $\operatorname{Aut}_F(\mathcal O)$-orbit.
		\begin{itemize}
			\item If $\dim_F T\leq 2$, then $T=F+Fc$ for some $c\in T$. By Proposition~\ref{prop: element octonion to quaternion}, $c$ lies in a quaternion subalgebra of $\mathcal O$, and hence so does $T$. So Proposition~\ref{prop: T contained in quaternion} applies.
			\item If $T$ is of type (3)(a), then $T\subseteq F+Fp+Fn\subseteq\mathcal Q$, and Proposition~\ref{prop: T contained in quaternion} applies with $\mathcal H=\mathcal Q$.
			\item If $T$ is of type (3)(b), then $A=a+r\omega$ and $B=b+s\omega$ with $a,b\in F$ and $r,s\in S$, and Proposition~\ref{prop:unified} applies.
			\item If $T$ is of type (4)(a), that is, $T$ is a quaternion subalgebra of $\mathcal O$, split or not, then Proposition~\ref{prop: T contained in quaternion} applies with $\mathcal H=T$.
			\item If $T$ is of type (4)(b), then $A=a+r\omega$ and $B=b+s\omega$ with $a,b\in\mathcal C\subseteq\mathcal Q$ and $r,s\in S$, and Proposition~\ref{prop:unified} applies.
		\end{itemize}
		Taking $r=s=0$ in Proposition~\ref{prop:unified}, one recovers Proposition~\ref{prop: T contained in quaternion} when $T$ is contained in the split quaternion subalgebra $\mathcal Q=M_2(F)$. However, Proposition~\ref{prop:unified} does not cover the case in which $T$ is a division quaternion subalgebra $\mathcal H$ of $\mathcal O$, not even after applying an automorphism of $\mathcal O$. Indeed, since $\overline uv=0$ for all $u,v\in S$, and $S$ is the set of matrices of $\mathcal Q$ with zero second row, it follows from \eqref{eq:doubling-mult} that
		$$
		(a+r\omega)(b+s\omega)=ab+(sa+r\overline b)\omega,\qquad sa+r\overline b\in S,
		$$
		for all $a,b\in\mathcal Q$ and $r,s\in S$. So $\mathcal Q+S\omega$ is a subalgebra of $\mathcal O$ and $\pi(a+r\omega)=a$ defines an algebra homomorphism $\pi:\mathcal Q+S\omega\to\mathcal Q$. If $\sigma(T)\subseteq\mathcal Q+S\omega$ for some $\sigma\in\operatorname{Aut}_F(\mathcal O)$, with $T$ a division algebra, then $\pi$ is injective on $\sigma(T)$ (its kernel is an ideal of the division algebra $\sigma(T)$), and $\pi(\sigma(T))$ is a $4$-dimensional subalgebra of $\mathcal Q=M_2(F)$, that is, $\pi(\sigma(T))=M_2(F)$. This is impossible, since $M_2(F)$ is not a division algebra. This is why Proposition~\ref{prop: T contained in quaternion}, whose proof uses the non-split part of Lemma~\ref{lem:quaternion nonvanishing}, is needed for this case.
	\end{remark}
	
	\begin{proof}[Proof of \cref{prop:unified}]
		For $u,v\in S=Fp+Fn$ we have $\overline uv=0$; in particular
		\begin{equation}
			\overline rr=\overline rs=\overline sr=\overline ss=0. \label{eq:vanishing-S}
		\end{equation}
		Since $N(u)=0$ for $u\in S$, $N(A)=N(a+r\omega)=N(a)-N(r)=N(a)$. Since $A$ is invertible, $a\in \mathcal Q$ is also invertible.
		
		\textbf{Case (1): $a\neq b$.} Let $C=a^{-1}b$ and $D=ba^{-1}$. By \cref{lem:quaternion nonvanishing}, applied to $a,b\in\mathcal Q$, there is $x_0\in \mathcal Q$ with $x_0-\overline{x_0}D$ invertible and $T^C_{x_0}$ invertible, where $T^C_{x_0}(Y)=x_0Y-CYx_0$. Put $K_{x_0}=x_0a-\overline{x_0}b=(x_0-\overline{x_0}D)a$, which is invertible. Let $Z=Z_0+Z_1\omega\in\mathcal O$ be arbitrary; take $x=x_0$, $y=y_0+y_1\omega$. Using \eqref{eq:doubling-mult},
		\begin{align*}
			A(xy)&=ax_0y_0+\overline{x_0}\,\overline{y_1}r+\bigl(y_1x_0a+r\overline{y_0}\,\overline{x_0}\bigr)\omega,\\
			B(yx)&=by_0x_0+x_0\overline{y_1}s+\bigl(y_1\overline{x_0}b+s\overline{x_0}\,\overline{y_0}\bigr)\omega,
		\end{align*}
		so $A(xy)-B(yx)=Z$ is the system
		\begin{align}
			Z_0&=ax_0y_0-by_0x_0+\overline{x_0}\,\overline{y_1}r-x_0\overline{y_1}s, \label{eq:Z0}\\
			Z_1&=y_1K_{x_0}+r\overline{y_0}\,\overline{x_0}-s\overline{x_0}\,\overline{y_0}. \label{eq:Z1}
		\end{align}
		From \eqref{eq:Z1}, $y_1=(Z_1-r\overline{y_0}\,\overline{x_0}+s\overline{x_0}\,\overline{y_0})K_{x_0}^{-1}$, and conjugating, $\overline{y_1}=\overline{K_{x_0}}^{-1}(\overline{Z_1}-x_0y_0\overline r+y_0x_0\overline s)$. Substituting into \eqref{eq:Z0} and using \eqref{eq:vanishing-S} to drop the cross terms,
		$$
		ax_0y_0-by_0x_0=Z_0-\overline{x_0}\,\overline{K_{x_0}}^{-1}\overline{Z_1}r+x_0\overline{K_{x_0}}^{-1}\overline{Z_1}s=:W_0,
		$$
		i.e., $aT^C_{x_0}(y_0)=W_0$; as $T^C_{x_0}$ is invertible, this determines $y_0$, and then $y_1$ from \eqref{eq:Z1}. (When $r=s=0$, this simplifies to $y_1=Z_1K_{x_0}^{-1}$ and $aT^C_{x_0}(y_0)=Z_0$, recovering the quaternion-contained case of the Remark above with no further computation.) So $\Phi$ is surjective.
		
		\textbf{Case (2): $a=b$.} Write $r=r_1p+r_2n$, $s=s_1p+s_2n$; since $A\neq B$ and $a=b$, $r\neq s$, so $(r_1-s_1,r_2-s_2)\neq(0,0)$. We exhibit $x\in\mathcal O$ with $\psi_x$ invertible. Write $a=\left(\begin{smallmatrix}a_{11}&a_{12}\\a_{21}&a_{22}\end{smallmatrix}\right)\in \mathcal Q=M_2(F)$.
		
		If $r_2\neq s_2$: take $x=1+m\omega$. Consider the ordered basis
		$$
		\mathcal B=(p,n,m,q,p\omega,n\omega,m\omega,q\omega)
		$$
		of $\mathcal O$. The matrix of $\psi_x$ with respect to $\mathcal B$ is
		{\footnotesize
			\[M_1=
			\begin{pmatrix}
				0 & -r_1-s_1 & 0 & 0 & 2a_{12} & -a_{11} & 0 & r_1-s_1\\
				0 & -r_2-s_2 & 0 & 0 & 0 & a_{12} & 0 & r_2-s_2\\
				s_1 & 0 & 0 & -r_1 & 2a_{22} & -a_{21} & -r_1+s_1 & 0\\
				s_2 & 0 & 0 & -r_2 & 0 & a_{22} & -r_2+s_2 & 0\\
				0 & 0 & -r_2+s_2 & r_1-s_1 & -r_2-s_2 & s_1 & 0 & 0\\
				r_2-s_2 & -r_1+s_1 & 0 & 0 & 0 & -r_2 & 0 & 0\\
				-a_{11} & -2a_{21} & 0 & a_{11} & 0 & 0 & 0 & 0\\
				-a_{12} & -2a_{22} & 0 & a_{12} & 0 & 0 & 0 & 0
			\end{pmatrix},\]}
		and a direct computation gives $\det(M_1)=4(r_2-s_2)^4\det(a)^2\neq0$.
		
		If $r_2=s_2$ (so $r_1\neq s_1$): take $x=1+q\omega$; the matrix of $\psi_x$ with respect to $\mathcal B$ is
		{\footnotesize
			\[M_2=
			\begin{pmatrix}
				r_1 & 0 & 0 & -s_1 & -a_{11} & 0 & 0 & r_1-s_1\\
				r_2 & 0 & 0 & -s_2 & a_{12} & -2a_{11} & 0 & r_2-s_2\\
				0 & 0 & r_1+s_1 & 0 & -a_{21} & 0 & -r_1+s_1 & 0\\
				0 & 0 & r_2+s_2 & 0 & a_{22} & -2a_{21} & -r_2+s_2 & 0\\
				0 & 0 & -r_2+s_2 & r_1-s_1 & r_1 & 0 & 0 & 0\\
				r_2-s_2 & -r_1+s_1 & 0 & 0 & -s_2 & r_1+s_1 & 0 & 0\\
				a_{21} & 0 & -2a_{11} & -a_{21} & 0 & 0 & 0 & 0\\
				a_{22} & 0 & -2a_{12} & -a_{22} & 0 & 0 & 0 & 0
			\end{pmatrix},\]}
		
		and $\det(M_2)=4(r_1-s_1)^4\det(a)^2\neq0$.
		
		In either case $\psi_x$ is invertible (note the formulas remain valid, and nonzero, when $a=\alpha I$ is scalar (i.e., the case $T=F+(Fp+Fn)\omega$) since then $\det(a)=\alpha^2\neq0$), so for every $Z\in\mathcal O$ there is $y\in\mathcal O$ with $A(xy)-B(yx)=Z$.
	\end{proof}

	\subsubsection{The Heisenberg-radical case}
	
	It remains to treat the case in which  $T$ is contained in the $\operatorname{Aut}_F(\mathcal O)$-orbit of $F+(n\mathcal O\cap\mathcal O n)$. By \cite[Example~3.7(b)]{MR4812310},
	$$
	R:=n\mathcal O\cap\mathcal O n=Fn+(Fn+Fq)\omega,
	$$
	is a three-dimensional Heisenberg subalgebra of $\mathcal O$.
	
	\begin{proposition}
		\label{prop:T Heisenberg case}
		Let $A,B\in\mathcal O\setminus\{0\}$ with $A$ invertible and $A\neq B$, and suppose $T=\langle A,B\rangle$ is in the $\operatorname{Aut}_F(\mathcal O)$-orbit of $F+(n\mathcal O\cap\mathcal O n)$. Then there exists $x_0\in\mathcal O$ such that $\psi_{x_0}$ is bijective. In particular, $\Phi$ is surjective.
	\end{proposition}
	\begin{proof}
		By \cref{lem: automorphism invariance} we may assume $T=F+R$ directly, and write
		$$
		A=a+\alpha n+(\beta n+\gamma q)\omega,\qquad B=b+\alpha'n+(\beta'n+\gamma' q)\omega,
		$$
		with $a,b,\alpha,\alpha',\beta,\beta',\gamma,\gamma'\in F$.
		
		Since $R$ is a Heisenberg subalgebra, it is totally isotropic ($N\equiv0$ on $R$) and contained in $1^\perp$ \cite[5.11(c)(i)]{MR4812310}: in particular $N(r,1)=0$ for $r\in R$, and polarizing $N(r+r')=N(r)+N(r')+N(r,r')$ gives $N(r,r')=0$ for all $r,r'\in R$. Hence, for $c\in F$, $r\in R$,
		$$
		N(c+r)=c^2N(1)+cN(r,1)+N(r)=c^2.
		$$
		In particular $N(A)=a^2\neq0$, as $A$ is invertible.
		
		\textbf{Case (1): $a\neq b$.} Then $A-B=(a-b)+r$ for some $r\in R$, so $N(A-B)=(a-b)^2\neq0$; by \cref{lem: N(A-B) nonzero}, $x_0=1$ works.
		
		\textbf{Case (2): $a=b$.} Then $(\alpha-\alpha',\beta-\beta',\gamma-\gamma')\neq(0,0,0)$; we exhibit $x\in\mathcal O$ with $\psi_x$ invertible. Denote by $[\psi_x]$ the matrix of $\psi_x$ with respect to the basis $\mathcal B=(p,n,m,q,p\omega,n\omega,m\omega,q\omega)$. Then
		$$
		\det[\psi_{1+m\omega}]=4a^4(\beta-\beta')^4,\qquad \det[\psi_{1+p\omega}]=4a^4(\gamma-\gamma')^4,
		$$
		$$
		\det[\psi_{m+\omega}]=-16a^4\bigl(a(\alpha-\alpha')+a(\gamma-\gamma')+\beta\gamma'-\beta'\gamma\bigr)^2,
		$$
		by direct computation. So: if $\beta\neq\beta'$, take $x=1+m\omega$; if $\beta=\beta'$ and $\gamma\neq\gamma'$, take $x=1+p\omega$; if $\beta=\beta'$ and $\gamma=\gamma'$ (so $\alpha\neq\alpha'$, and the third determinant reduces to $-16a^6(\alpha-\alpha')^2$), take $x=m+\omega$. In each case $\det[\psi_x]\neq0$, since $a\neq0$ and $\operatorname{char}F\neq2$, so $x_0:=x$ has $\psi_{x_0}$ bijective.
	\end{proof}

	\subsubsection{The main theorem for split octonion algebras}
	
	We are finally in a position to state and prove the main theorem of this section.
	\begin{theorem}
		\label{thm: split octonion surjectivity-invertible}
		Let $\mathcal O$ be a split octonion algebra over $F$, $\operatorname{char}F\neq2$. If $A\in\mathcal O^\times$ and $B\in\mathcal O\setminus\{0\}$ with $A\neq B$, then there exists $x_0\in\mathcal O$ such that $\psi_{x_0}:y\mapsto\Phi_{A,B}(x_0,y)$ is bijective. In particular, $\Phi_{A,B}$ is surjective.
	\end{theorem}
	\begin{proof}
		Let $T$ be the subalgebra of $\mathcal O$ generated by $A$ and $B$. As noted above, $T$ is unital, associative, and $\dim T\leq4$, so \cref{lem: unitary subalgebras} applies. If $T$ is not of Heisenberg-radical type (case (4)(c)), then, as explained in Remark~\ref{rem:unified-coverage}, up to the action of $\operatorname{Aut}_F(\mathcal O)$, either $T$ is contained in a quaternion subalgebra of $\mathcal O$ (split or not), and Proposition~\ref{prop: T contained in quaternion} gives $x_0$ with $\psi_{x_0}$ bijective, or $A$ and $B$ satisfy the hypotheses of Proposition~\ref{prop:unified}, which gives the same conclusion. If $T$ is of Heisenberg-radical type, \cref{prop:T Heisenberg case} does the same. Since the existence of such an $x_0$ is invariant under $\operatorname{Aut}_F(\mathcal O)$ (\cref{lem: automorphism invariance}), this covers every case.
	\end{proof}
	
	\section{Both coefficients are zero divisors}
	\label{sec:zero-divisors}
	
	In this section, we consider the case where both $A$ and $B$ are nonzero zero divisors, and we prove Theorem~\ref{thm:main}.
	
	\begin{proposition}
		\label{prop:both zero divisors}
		Let $A$ and $B$ be nonzero zero divisors of $\mathcal O$. Define the map $\Phi: \mathcal O \times \mathcal O\longrightarrow \mathcal O$ given by $\Phi(x,y)=A(xy)-B(yx)$ and $W=\operatorname{span}_F\operatorname{Im}(\Phi)$. Then
		$$
		\dim_F W=
		\begin{cases}
			6,&\text{if $A$ and $B$ are linearly independent and $N(A,B)=0$},\\
			4,&\text{if $A$ and $B$ are linearly dependent and $N(A,B)=0$},\\
			8,&\text{if $N(A,B)\neq 0$}.
		\end{cases}
		$$
		Moreover, if $N(A,B)=0$, then $W=A\mathcal O+B\mathcal O$.
	\end{proposition}
	\begin{proof}
		First we consider, $N(A,B)=0$. Since $N(\cdot,\cdot)$ is nondegenerate, $\dim W+ \dim W^\perp =8$. We claim that $W^\perp=A\mathcal O \cap B\mathcal O$.
		Let $z\in W^\perp$. For all $x,y\in \mathcal O$, 
		\begin{align*}
			N(z,A(xy)-B(yx))  = 0
			\iff N(\overline{x}(\overline{A}z)-(\overline{B}z)\overline{x},y)  = 0
		\end{align*}
		Since $N(\cdot,\cdot)$ is nondegenerate, $\overline{x}(\overline{A}z)-(\overline{B}z)\overline{x}=0$ for all $x\in \mathcal O$. If $x=1$, $\overline{A}z=\overline{B}z:=c\in \mathcal O$. Therefore, for all $x\in \mathcal O$, $\overline{x}c=c\overline{x}$. Thus, $c\in F1$. Since $N(A)=0$, we have $c=0$. Therefore, $z\in \ker(L_{\overline{A}})\cap \ker(L_{\overline{B}})$. Conversely, if $z\in \ker(L_{\overline{A}})\cap \ker(L_{\overline{B}})$
		then $\overline{A}z=\overline{B}z=0$, and the preceding calculation shows that $N(z,A(xy)-B(yx))=0$
		for all $x,y\in\mathcal O$. Thus $z\in W^\perp$, and consequently
		$$
		W^\perp
		=
		\ker(L_{\overline{A}})\cap
		\ker(L_{\overline{B}}).
		$$
		Since $A$ and $B$ are isotropic elements, by Proposition~\ref{Image_of_L_A}, $A\mathcal O=\ker(L_{\overline{A}})$ and $B\mathcal O=\ker(L_{\overline{B}})$ (see also \cite[Theorem~3.6(a)]{MR4812310}) as $\dim A\mathcal O=4$ and $\dim\operatorname{Im}(L_{\overline{A}})=4$, and same holds for $B$. Thus, $W^\perp = A\mathcal O\cap B \mathcal O$. Since $A\mathcal O$ and $B\mathcal O$ are totally isotropic subspaces of dimension $4$, we have $(A\mathcal O)^\perp=A\mathcal O$ and $(B\mathcal O)^\perp=B\mathcal O$, and so $W=(W^\perp)^\perp=(A\mathcal O\cap B\mathcal O)^\perp=(A\mathcal O)^\perp+(B\mathcal O)^\perp=A\mathcal O+B\mathcal O$. By \cite[Theorem 3.6]{MR4812310}, if $A$ and $B$ are linearly independent, then $\dim W^\perp= 2$. If $A$ and $B$ are linearly dependent, then $W^\perp=A\mathcal O$, which is a $4$-dimensional subspace of $\mathcal O$. Thus, we get $\dim W=6$ if $A$ and $B$ are linearly independent and $\dim W=4$ if $A$ and $B$ are linearly dependent.
		
		If $N(A,B)\neq 0$, then $\dim W=8$ follows from Lemma~\ref{lem: N(A-B) nonzero}.
	\end{proof}
	
	We now show that, in the first case of Proposition~\ref{prop:both zero divisors}, the image of $\Phi$ coincides with its linear span. For this, we will reduce the pair $(A,B)$, using automorphisms of $\mathcal O$, to some normal forms. Recall that $p=E_{11}$, $q=E_{22}=1-p$, $n=E_{12}$ and $m=E_{21}$.
	
	Let us consider the Peirce decomposition of $\mathcal O$ with respect to the idempotent $p$:
	$$\mathcal O = p\mathcal Op\oplus q\mathcal Oq\oplus p\mathcal Oq\oplus q\mathcal Op.$$
	Notice that, since $\mathcal O$ is alternative (in particular, flexible) and $q = 1-p$, we have, for instance, $p(xq) = (px)q$ for each $x\in \mathcal O$, so the above summands are well defined. Direct computations using \eqref{eq:doubling-mult} show that $p\mathcal O p = F p$, $q\mathcal O q = F q$,
	$$
	U := p\mathcal O q  = F n \oplus F (p\omega)\oplus F (m \omega),\qquad V := q\mathcal O p  = F m \oplus F (q\omega)\oplus F (n \omega).
	$$
	We consider the basis $\{u_1,u_2,u_3\} = \{n, p\omega, m\omega\}$ of $U$ and the basis $\{v_1,v_2,v_3\} = \{m, q\omega, -n\omega\}$ of $V$. We identify elements of $U$ and $V$ with their coordinates in these bases, i.e., elements of $U$ and $V$ are in one-to-one correspondence with elements of $F^3$, and we will use these representations here. In particular, $u_i$ and $v_i$ are both identified with the $i$-th canonical vector $e_i$ of $F^3$. Direct computations give
	\[u_iv_j = \delta_{ij}p, \qquad v_iu_j = \delta_{ij}q, \qquad  u_iu_j = -\sum_k\varepsilon_{ijk}v_k, \qquad v_iv_j = \sum_k \varepsilon_{ijk}u_k,\]
	where $\delta$ is the Kronecker symbol and $\varepsilon$ is the alternating symbol with $\varepsilon_{123}=1$. As a consequence, by bilinearity, for any $u, u'\in U$ and $v, v'\in V$, we have
	\begin{equation}
		uv = (u^Tv) p, \qquad vu = (u^Tv) q, \qquad uu' = -(u\times u')\in V, \qquad vv' = v\times v'\in U, \label{eq:Peirce products}
	\end{equation}
	where $\times$ denotes the usual cross product in $F^3$. For instance, $uu' = \sum_{i,j}a_ia'_ju_iu_j = -\sum_k\left(\sum_{i,j}\varepsilon_{ijk}a_ia'_j\right)v_k$, where $a_i$ and $a'_j$ denote the coordinates of $u$ and $u'$, and $\sum_{i,j}\varepsilon_{ijk}a_ia'_j$ is the $k$-th coordinate of $u\times u'$. In particular, we can identify $V$ with the dual space of $U$ through the pairing $v\mapsto \tilde{v}$, where $\tilde{v}:U\longrightarrow F$ is given by $\tilde{v}(u) = u^Tv$.
	
	The Peirce decomposition also simplifies the expressions of the norm and the trace of an element. In fact, if $x = \alpha p +\beta q +u +v$, with $\alpha, \beta\in F$, $u\in U$ and $v\in V$, then
	\begin{equation}
		N(x) = \alpha \beta -u^Tv\qquad\text{and}\qquad \operatorname{Tr}(x) = \alpha +\beta. \label{eq:Peirce norm}
	\end{equation}
	
	We now define, for each $g\in SL_3(F)$, an automorphism $\sigma_g$ of $\mathcal O$. Here $g^{-T}$ denotes the transpose of the inverse of $g$.
	
	\begin{lemma}
		\label{lem:SL3 automorphisms}
		Let $g\in SL_3(F)$, and let $\sigma_g:\mathcal O\longrightarrow\mathcal O$ be the linear map given by $\sigma_g(p) = p$, $\sigma_g(q) = q$, $\sigma_g(u) =gu$ and $\sigma_g(v) = g^{-T}v$, for $u\in U$ and $v\in V$. Then $\sigma_g$ is an automorphism of $\mathcal O$.
	\end{lemma}
	\begin{proof}
		Notice that $\sigma_g$ is a bijective linear map which preserves the Peirce components of $\mathcal O$. Since $\sigma_g$ fixes $p$ and $q$ and preserves each Peirce component, it preserves the products involving $p$ or $q$ (for instance, $pu = u$ and $up = 0$ for every $u\in U$). For $u\in U$ and $v\in V$, by \eqref{eq:Peirce products} we have $\sigma_g(u)\sigma_g(v) = (gu)(g^{-T}v) = ((gu)^T(g^{-T}v))p = (u^Tg^{T}g^{-T}v)p = (u^Tv)p = \sigma_g(uv)$, and similarly $\sigma_g(v)\sigma_g(u) = \sigma_g(vu)$. For the products of two elements of $U$, or of two elements of $V$, we use the identity
		\[(Mx)\times(My) = \det(M)\,M^{-T}(x\times y),\]
		which holds for every $M\in GL_3(F)$ and $x,y\in F^3$. Indeed, for every $w\in F^3$ we have $(Mw)^T((Mx)\times(My)) = \det[Mw\ Mx\ My] = \det(M)\det[w\ x\ y] = \det(M)\, w^T(x\times y)$, that is, $w^T \left(M^T((Mx)\times(My))\right) = w^T\left(\det(M)(x\times y)\right)$ for every $w$, and so $M^T((Mx)\times(My)) = \det(M)(x\times y)$. Since $\det g = 1$, we obtain, for $u,u'\in U$,
		\[\sigma_g(u)\sigma_g(u') = -(gu)\times(gu') = -g^{-T}(u\times u') = \sigma_g(-(u\times u')) = \sigma_g(uu'),\]
		and, applying the identity to $M = g^{-T}$ (which also has determinant $1$ and satisfies $M^{-T} = g$), for $v,v'\in V$,
		\[\sigma_g(v)\sigma_g(v') = (g^{-T}v)\times(g^{-T}v') = g(v\times v') = \sigma_g(v\times v') = \sigma_g(vv').\]
		By bilinearity, $\sigma_g(xy) = \sigma_g(x)\sigma_g(y)$ for all $x,y\in\mathcal O$.
	\end{proof}
	
	\begin{lemma}
		\label{lem:SL2 automorphisms}
		Let $T\in SL_2(F)$, and let $\tau_T:\mathcal O\longrightarrow\mathcal O$ be given by $\tau_T(a + r\omega) = a + (T^{-1}r) \omega$, for $a, r\in \mathcal Q=M_2(F)$. Then $\tau_T$ is an automorphism of $\mathcal O$.
	\end{lemma}
	\begin{proof}
		The map $\tau_T$ is clearly linear and bijective (its inverse is $\tau_{T^{-1}}$), and one just needs to show that it respects the product in $\mathcal O$. By \eqref{eq:doubling-mult}, we have
		\begin{align*}\tau_T(a+r\omega)\tau_T(b+s\omega) &= (a+(T^{-1}r)\omega)(b+(T^{-1}s)\omega)\\ &= ab+\overline{T^{-1}s}\,(T^{-1}r) + \left((T^{-1}s) a+ (T^{-1}r) \overline b\right)\omega.\end{align*}
		Since $\overline{T^{-1}s} = \overline{s}\ \overline{T^{-1}}$ and $\overline{T^{-1}}\,T^{-1} = \det(T^{-1})I = I$, we have $\overline{T^{-1}s}\,(T^{-1}r) = \overline{s}r$. Also, $(T^{-1}s) a+ (T^{-1}r) \overline b = T^{-1}(sa+r\overline b)$. Hence the above equals
		\[ ab+\overline{s}r +\left(T^{-1}(sa+r\overline{b})\right)\omega = \tau_T((a+r\omega)(b+s\omega)),\]
		and $\tau_T$ is an automorphism.
	\end{proof}
	
	\begin{proposition}
		\label{prop:independent zero divisors}
		Let $A$ and $B$ be linearly independent elements of $\mathcal O$ such that $N(A) = N(B) = N(A,B) = 0$. Then the image of $\Phi_{A,B}$ is $A\mathcal O+ B\mathcal O$, which is a $6$-dimensional subspace of $\mathcal O$.
	\end{proposition}
	
	\begin{proof}
		By Proposition~\ref{prop:both zero divisors}, the linear span of $\operatorname{Im}(\Phi_{A,B})$ is $A\mathcal O+ B\mathcal O$, which has dimension $6$. So it is enough to show that there exists $x_0\in \mathcal O$ such that the linear map $\psi_{x_0}(y)=\Phi_{A,B}(x_0,y)$ has rank $6$. Indeed, in this case
		\[A\mathcal O+B\mathcal O = \operatorname{Im}(\psi_{x_0})\subseteq\operatorname{Im}(\Phi_{A,B})\subseteq A\mathcal O+B\mathcal O.\]
		
		To that end, by Lemma~\ref{lem: automorphism invariance}, we may restrict our attention to representatives of the orbits of pairs $(A,B)$ under the action of $\operatorname{Aut}_F(\mathcal O)$. We may also interchange $A$ and $B$, since $\Phi_{A,B}(x,y) = -\Phi_{B,A}(y,x)$, and multiply both $A$ and $B$ by the same nonzero scalar $\lambda$, since $\Phi_{\lambda A,\lambda B} = \lambda\Phi_{A,B}$ and the image of $\Phi_{A,B}$ is closed under multiplication by scalars. These operations preserve the hypotheses on $A$ and $B$, and the subspace $A\mathcal O+ B\mathcal O$. (If we interchange $A$ and $B$, the element $x_0$ obtained gives a map $x\mapsto \Phi_{A,B}(x,x_0)$ of rank $6$, which is enough for our purposes.) We will consider three cases, which will lead us to five normal forms for the pair $(A,B)$. Then, for each normal form, we will obtain the desired result.
		
		\textbf{Case (1): At least one of $A$ and $B$ has nonzero trace.} Interchanging $A$ and $B$ if necessary, we may assume $\operatorname{Tr}(A)\neq 0$. The Cayley--Hamilton identity then gives $A^2 = \operatorname{Tr}(A)A$. By replacing $A$ and $B$ by $A/\operatorname{Tr}(A)$ and $B/\operatorname{Tr}(A)$, respectively, we may assume that $A$ has trace $1$ and norm $0$. By \cite[Lemma~4.4]{MR4812310}, $A$ is in the same orbit as $p$. So we may assume $A = p$.
		
		Now write $B = \alpha p+\beta q + u + v$, with $\alpha,\beta\in F$, $u\in U$ and $v\in V$. Since $N(p) = N(B) = 0$, we have $N(p+B) = N(p,B) = 0$. Applying \eqref{eq:Peirce norm} to $B$ and to $p+B = (1+\alpha)p+\beta q+u+v$, we obtain
		\[0 = N(p+B) - N(B) = (1+\alpha)\beta - u^Tv - (\alpha\beta - u^Tv) = \beta,\]
		that is, the $q$-component of $B$ is zero. As a consequence, we may write $B = \alpha p+ u +v$, and $N(B) = -u^Tv = 0$.
		
		We will show below that, in the present case, using the automorphisms $\sigma_g$ of Lemma~\ref{lem:SL3 automorphisms}, it is enough to consider only three normal forms for $(A,B)$. Notice that $\sigma_g(A) = A$ and $\sigma_g(B) = \alpha p + gu + g^{-T}v$ for every $g\in SL_3(F)$. The case $u = v = 0$ does not occur, since $A$ and $B$ are linearly independent. Assume now $u = 0$ and $v\neq 0$. Since $SL_3(F)$ acts transitively on $F^3\setminus\{0\}$, we can choose $g$ such that $\sigma_g(v) = g^{-T}v = v_1 = m$. In particular, we have reduced the pair $(A,B)$ to $(p, \alpha p+m)$. Similarly, if $u\neq 0$ and $v = 0$, we can choose $g$ with $gu = u_1 = n$, and we obtain $(p, \alpha p+n)$.
		Finally, if $u\neq 0$ and $v\neq 0$, let $z\in U$ with $z^Tv = 0$, linearly independent of $u$ (this is possible since the kernel of the linear functional $\tilde v$ has dimension $2$ and contains $u$), and let $w\in U$ be such that $w^Tv = 1$.
		Now, the set $\{u, w, z\}$ is a basis of $U$ and the matrix $P = (u \ w \ z)$ is invertible. Let $c = \det P$ and define $g = \operatorname{diag}(1,1,c)P^{-1}$. Then $\det g = c\det(P)^{-1} = 1$, that is, $g\in SL_3(F)$. Let us show that $\sigma_g(B) = \alpha p + n + q\omega$, i.e., that $gu = n$ and $g^{-T}v = q\omega$.
		
		It is immediate that the matrix $P$ above satisfies $Pu_1 = u$, $Pu_2 = w$ and $Pu_3 = z$. This implies that $gu = u_1$, $g w =u_2$ and $gz = cu_3$. Since $u_1 = n$, it remains to prove that $g^{-T}v = v_2 = q\omega$.
		From $g = \operatorname{diag}(1,1,c)P^{-1}$ we obtain $g^{-1} = P\,\operatorname{diag}(1,1,c^{-1})$, and so
		\[g^{-T} = \left(g^{-1}\right)^T = \operatorname{diag}(1,1,c^{-1})\,P^T.\]
		The rows of $P^T$ are $u^T$, $w^T$ and $z^T$. Hence, since $u^Tv = 0$, $w^Tv = 1$ and $z^Tv = 0$,
		\[g^{-T}v = \operatorname{diag}(1,1,c^{-1})\begin{pmatrix} u^Tv\\ w^Tv\\ z^Tv\end{pmatrix}
		= \operatorname{diag}(1,1,c^{-1})\begin{pmatrix} 0\\ 1\\ 0\end{pmatrix}
		= \begin{pmatrix} 0\\ 1\\ 0\end{pmatrix} = v_2 = q\omega.\]
		Therefore $\sigma_g(A) = p$ and $\sigma_g(B) = \alpha p + gu + g^{-T}v = \alpha p + n + q\omega$, and we have reduced the pair $(A,B)$ to $(p, \alpha p + n + q\omega)$.
		
		Summarizing, in Case (1) the pair $(A,B)$ is in the orbit of one of the following pairs:
		\[(p,\ \alpha p+n),\qquad (p,\ \alpha p+m),\qquad (p,\ \alpha p+n+q\omega),\qquad \alpha\in F.\]
		
		\textbf{Case (2): $\operatorname{Tr}(A) = \operatorname{Tr}(B) = 0$ and $AB = 0$.} In this case, the Cayley--Hamilton identity gives $A^2 = B^2 = (A+B) ^2 = 0$. As a consequence, $AB+BA = 0$ and we also obtain $BA = 0$. By \cite[Corollary~5.10]{MR4812310}, $\operatorname{Aut}_F(\mathcal O)$ acts transitively on ordered pairs of distinct $1$-dimensional subspaces generated by elements with norm and trace zero, and with zero product, and the subspaces spanned by $A$ and $B$ satisfy these conditions. So do the subspaces spanned by $p\omega$ and $n\omega$, since $N(p\omega) = N(n\omega) = 0$, $\operatorname{Tr}(p\omega) = \operatorname{Tr}(n\omega) = 0$ and $(p\omega)(n\omega) = 0$. As a consequence, there exists an automorphism $\sigma$ of $\mathcal O$ such that $\sigma(A) = \lambda p\omega$ and $\sigma(B) = \mu n\omega$, for some $\lambda,\mu\in F^\times$, and multiplying both coefficients by $\lambda^{-1}$, we can reduce the pair $(A, B)$ to the normal form $(p\omega, \alpha n\omega)$, with $\alpha = \mu\lambda^{-1} \in F^\times$.
		
		\textbf{Case (3): $\operatorname{Tr}(A) = \operatorname{Tr}(B) = 0$ and $AB\neq0$.} Recall that $N(A) = N(B) = N(A,B) = 0$, and that this implies $N(A+B) = 0$ as well. Also, $\operatorname{Tr}(A+B) = 0$, and the Cayley--Hamilton identity implies $A^2 = B^2 = (A+B)^2 = 0$. This implies $AB = -BA$. Now let $C = AB$. Then $AC = CA = BC = CB = 0$. This is an immediate consequence of the alternative laws. For instance, $CA  =A (BA) = A(-AB) = -A^2B = 0$.
		Also, $\operatorname{Tr}(C) = \operatorname{Tr}(AB) = \operatorname{Tr}(BA) = \operatorname{Tr}(-AB) = -\operatorname{Tr}(C)$ and since the characteristic of $F$ is not $2$, we have $\operatorname{Tr}(C) = 0$. Moreover, $N(C) = N(A)N(B) = 0$ and $C\neq 0$, so $C$ is not central. Since $n$ is also a noncentral element with $N(n) = \operatorname{Tr}(n) = 0$, by \cite[Lemma~4.4]{MR4812310} there exists an automorphism $\sigma $ of $\mathcal O$ such that $\sigma(C) = n$. In particular, $\sigma(A)\sigma(B) = \sigma(AB) = n$, and the pair $(\sigma(A),\sigma(B))$ satisfies the same hypotheses as $(A,B)$. So, replacing $(A,B)$ by $(\sigma(A),\sigma(B))$, we may assume $AB = n$. In this case, both $A$ and $B$ lie in the two-sided annihilator of $n$, which, as a direct computation shows, is equal to
		\[\{z\in \mathcal O\mid nz = zn = 0\} = Fn + (Fn+Fq)\omega.\]
		In particular, we can write $A = \alpha n + (\gamma n+\delta q)\omega$ and $B = \beta n + (\gamma'n+\delta'q)\omega$, for some $\alpha, \beta, \gamma, \gamma', \delta, \delta'\in F$.
		
		Using the above expressions for $A$ and $B$, we get $AB = (\gamma \delta'-\gamma'\delta)n$. Since $AB = n$, we must have $\gamma \delta'-\gamma'\delta = 1$. In particular, the matrix $T = \begin{pmatrix}
			\gamma & \gamma'\\
			\delta & \delta '
		\end{pmatrix}$ lies in $SL_2(F)$, and we may consider the automorphism $\tau_T$ of Lemma~\ref{lem:SL2 automorphisms}. Now notice that $T^{-1} = \overline T = \begin{pmatrix}
			\delta' & -\gamma'\\
			-\delta & \gamma
		\end{pmatrix}$ and this implies $T^{-1}n = \delta 'n - \delta q$ and $T^{-1} q = -\gamma' n + \gamma q$. Since $\gamma\delta'-\gamma'\delta = 1$, this gives $\tau_T((\gamma n+\delta q)\omega) = (\gamma\delta'-\gamma'\delta)n\omega = n\omega$ and $\tau_T((\gamma' n+\delta' q)\omega) = (\gamma\delta'-\gamma'\delta)q\omega = q\omega$. Moreover, $\tau_T$ fixes $n$. So, replacing $(A,B)$ by $(\tau_T(A),\tau_T(B))$, the pair $(A,B)$ can be reduced to the normal form
		\[(A,B) = (\alpha n + n\omega , \beta n+ q\omega).\]
		
		As a consequence, we have only five normal forms for $(A,B)$ to consider. We just need to show that, for each of them, there exists $x_0$ such that the linear map $\psi_{x_0}=\Phi_{A,B}(x_0,\cdot)$ has rank $6$. Since $\psi_{x_0}$ is linear, its image is a subspace of $\mathcal O$, and for each normal form we exhibit six linearly independent elements in it. All the values below follow directly from \eqref{eq:doubling-mult}.
		
		\textit{The pair $(p,\ \alpha p+n)$.} Take $x_0 = p+m$. Then
		\[
		\psi_{x_0}(m) = -p,\quad \psi_{x_0}(n) = n-\alpha p,\quad \psi_{x_0}(p\omega) = p\omega,
		\]
		\[
		\psi_{x_0}(n\omega) = n\omega+(\alpha+1)p\omega,\quad \psi_{x_0}(m\omega) = m\omega,\quad \psi_{x_0}(q\omega) = q\omega+(\alpha+1)m\omega .
		\]
		Hence the image of $\psi_{x_0}$ contains $p$, $n$, $p\omega$, $n\omega$, $m\omega$ and $q\omega$.
		
		\textit{The pair $(p,\ \alpha p+m)$.} Take $x_0 = n+m$. Then
		\[
		\psi_{x_0}(m) = p,\quad \psi_{x_0}(q) = n,\quad \psi_{x_0}(p) = -\alpha n-q,
		\]
		\[
		\psi_{x_0}(n) = -\alpha p-m,\quad \psi_{x_0}(p\omega) = p\omega,\quad \psi_{x_0}(m\omega) = m\omega .
		\]
		Hence the image of $\psi_{x_0}$ contains $p$, $n$, $q$, $m$, $p\omega$ and $m\omega$.
		
		\textit{The pair $(p,\ \alpha p+n+q\omega)$.} Take $x_0 = p+m$. Then
		\[
		\psi_{x_0}(p\omega) = p\omega,\quad \psi_{x_0}(m\omega) = m\omega,\quad \psi_{x_0}(q\omega) = q\omega+(\alpha+1)m\omega,
		\]
		\[
		\psi_{x_0}(m) = -p+m\omega,\quad \psi_{x_0}(n) = -\alpha p+n-q\omega,\quad \psi_{x_0}(n\omega) = n+q+(\alpha+1)p\omega+n\omega .
		\]
		Using the previous elements successively, we see that the image of $\psi_{x_0}$ contains $p\omega$, $m\omega$, $q\omega$, $p$, $n$ and $q+n\omega$, which are linearly independent.
		
		\textit{The pair $(p\omega,\ \alpha\, n\omega)$, with $\alpha\neq0$.} Take $x_0 = p+q\omega$. Then
		\[
		\psi_{x_0}(n\omega) = n\omega,\quad \psi_{x_0}(q) = -\alpha\,n,\quad \psi_{x_0}(m\omega) = -m,\quad \psi_{x_0}(p) = p-\alpha\,n\omega,
		\]
		\[
		\psi_{x_0}(p\omega) = p\omega-\alpha\,n\omega,\quad \psi_{x_0}(m) = m+\alpha\,q+\alpha\,p\omega .
		\]
		Since $\alpha\neq 0$, the image of $\psi_{x_0}$ contains $n\omega$, $n$, $m$, $p$, $p\omega$ and $q$.
		
		\textit{The pair $(\alpha n+n\omega,\ \beta n+q\omega)$.} Take $x_0 = n+m$. Then
		\[
		\psi_{x_0}(m\omega) = n,\quad \psi_{x_0}(m) = n\omega-\beta n,\quad \psi_{x_0}(n) = \alpha n-q\omega,
		\]
		\[
		\psi_{x_0}(n\omega) = q+(\alpha+\beta)n\omega,\quad \psi_{x_0}(p) = \alpha p-p\omega,\quad \psi_{x_0}(q) = m\omega-\beta p .
		\]
		Hence the image of $\psi_{x_0}$ contains $n$, $n\omega$, $q\omega$, $q$, $\alpha p-p\omega$ and $m\omega-\beta p$. These six elements are linearly independent, since each of $n$, $n\omega$, $q\omega$, $q$, $p\omega$ and $m\omega$ appears in exactly one of them.
		
		In all cases $\psi_{x_0}$ has rank at least $6$, and the proof is complete.
	\end{proof}
	
	We are now ready to prove our main result.
	
	\begin{proof}[Proof of Theorem~\ref{thm:main}]
		Since $A\neq0$, either at least one of $A$ and $B$ is invertible, or $A$ is a nonzero zero divisor and $N(B)=0$.
		
		Assume first that $A$ or $B$ is invertible. If $A=B$, then $\operatorname{Im}(\Phi)=A\mathcal O_0=A^\perp$, by Lemma~\ref{lem: A=B image} and Proposition~\ref{prop: A=B invertible case}. This is a hyperplane of $\mathcal O$, so $\Phi$ is not surjective, and we obtain (i). If $A\neq B$, then $\Phi$ is surjective. Indeed, since $\Phi_{A,B}(x,y)=-\Phi_{B,A}(y,x)$, we may assume that $A$ is invertible. If $\mathcal O$ is a division algebra, then $A-B$ is invertible, and $\Phi$ is surjective by Lemma~\ref{lem: N(A-B) nonzero}. If $\mathcal O$ is split, then $\Phi$ is surjective by Lemma~\ref{lem: B is zero} if $B=0$, and by Theorem~\ref{thm: split octonion surjectivity-invertible} if $B\neq0$.
		
		Assume now that $A$ is a nonzero zero divisor and $N(B)=0$. Suppose first that $B=\alpha A$, with $\alpha\in F$. If $\alpha=1$, then $\operatorname{Im}(\Phi)=A\mathcal O_0=A\mathcal O$, by Lemma~\ref{lem: A=B image} and Proposition~\ref{prop: zero-divisor commutator}. If $\alpha\neq1$, then $(A-B)\mathcal O=A\mathcal O=A\mathcal O+B\mathcal O$, and $\operatorname{Im}(\Phi)=A\mathcal O$ by Lemma~\ref{lem:sandwich}. By Proposition~\ref{Image_of_L_A}, $A\mathcal O$ is a totally isotropic subspace of dimension $4$, which is maximal, since $N$ is nondegenerate and $\dim\mathcal O=8$. In particular, $\Phi$ is not surjective, and we obtain (ii). Suppose now that $A$ and $B$ are linearly independent. If $N(A,B)\neq0$, then $N(A-B)=N(A)+N(B)-N(A,B)=-N(A,B)\neq0$, and $\Phi$ is surjective by Lemma~\ref{lem: N(A-B) nonzero}. If $N(A,B)=0$, then $\operatorname{Im}(\Phi)=A\mathcal O+B\mathcal O$, which has dimension $6$, by Proposition~\ref{prop:independent zero divisors}, so $\Phi$ is not surjective, and we obtain (iii).
		
		Since $N(A,B)=2\alpha N(A)=0$ when $B=\alpha A$ and $N(A)=0$, the cases in which $\Phi$ is surjective are exactly those in (a) and (b). Finally, in all cases $\operatorname{Im}(\Phi)$ is one of $\mathcal O$, $A^\perp$, $A\mathcal O$ or $A\mathcal O+B\mathcal O$, which are vector subspaces of $\mathcal O$.
	\end{proof}
	
	\section{The quaternion case}
	\label{sec:quaternion}
	
	In this section we determine the image of the analogous map on a quaternion algebra. For $Q=M_2(K)$, with $K$ an algebraically closed field, Panja, Saini and Singh proved that the image is always a vector space \cite[Theorem~B]{PanjaSainiSingh2025}. Theorem~\ref{thm:quaternion} below recovers this for every field $F$ with $\operatorname{char}F\neq2$, covers division quaternion algebras as well, and identifies the image in each case.
	
	Let $Q$ be a quaternion algebra over $F$ and $Q_0=\{x\in Q\mid \operatorname{Tr}(x)=0\}$. For $A,B\in Q$ we consider $\Phi^Q_{A,B}(x,y)=Axy-Byx$ and, for $x\in Q$, the linear map $\psi^Q_x(y)=\Phi^Q_{A,B}(x,y)$. As in Lemma~\ref{lem:sandwich},
	\begin{equation}
		(A-B)Q\ \subseteq\ \operatorname{Im}(\Phi^Q_{A,B})\ \subseteq\ AQ+BQ, \label{eq:sandwich-Q}
	\end{equation}
	and, as before, we may replace $(A,B)$ by $(\sigma(A),\sigma(B))$, for any automorphism $\sigma$ of $Q$.
	
	If $Q$ is split, we identify $Q=M_2(F)$, so that $N=\det$ and $\operatorname{Tr}=\operatorname{tr}$. The nonzero zero divisors of $Q$ are the matrices of rank one, and for such a matrix $A$, regarded as a linear map of $F^2$, we have
	\begin{equation}
		AQ=\{X\in M_2(F)\mid XF^2\subseteq AF^2\}. \label{eq:rank-one-ideals}
	\end{equation}
	Hence, if $A$ and $B$ are nonzero zero divisors, then either $AQ=BQ$, or $AQ\cap BQ=0$ and $AQ+BQ=Q$.
	
	\begin{theorem}
		\label{thm:quaternion}
		Let $Q$ be a quaternion algebra over a field $F$ with $\operatorname{char}F\neq2$, and let $A,B\in Q$ with $A\neq0$. Then $\operatorname{Im}(\Phi^Q_{A,B})$ is a vector subspace of $Q$. More precisely:
		\begin{enumerate}
			\item if $A$ or $B$ is invertible and $A\neq B$, then $\operatorname{Im}(\Phi^Q_{A,B})=Q$;
			\item if $A=B$ is invertible, then $\operatorname{Im}(\Phi^Q_{A,B})=AQ_0=A^\perp$;
			\item if $N(A)=N(B)=0$, then $\operatorname{Im}(\Phi^Q_{A,B})=AQ+BQ$.
		\end{enumerate}
	\end{theorem}
	
	\begin{proof}
		First let $A=B$. Since $\operatorname{Tr}(xy)=\operatorname{Tr}(yx)$, Corollary~\ref{cor: trace zero quaternion} gives $\operatorname{Im}(\Phi^Q_{A,A})=AQ_0$. If $A$ is invertible, then $AQ_0=A^\perp$, as in the proof of Proposition~\ref{prop: A=B invertible case}. If $N(A)=0$, then $AQ_0=AQ$, as in the proof of Proposition~\ref{prop: zero-divisor commutator}. This proves (2), and (3) when $A=B$.
		
		So let $A\neq B$. If $Q$ is a division algebra, then $A-B$ is invertible, and $\Phi^Q_{A,B}$ is surjective by \eqref{eq:sandwich-Q}. Hence we may assume $Q=M_2(F)$.
		
		(1) Since $\Phi^Q_{A,B}(x,y)=-\Phi^Q_{B,A}(y,x)$, we may assume that $A$ is invertible. Let $C=A^{-1}B\neq1$. Then $\psi^Q_x=L_A\circ T^C_x$, and by Lemma~\ref{lem:quaternion nonvanishing} there exists $x_0\in Q$ with $\det(T^C_{x_0})\neq0$. Hence $\psi^Q_{x_0}$ is bijective.
		
		(3) If $BQ\subseteq AQ$ (for instance, if $B=0$), then $A-B$ is a nonzero element of $AQ$. So $N(A-B)=0$ (since $N(AX)=N(A)N(X)=0$) and $(A-B)Q\subseteq AQ$, and both have dimension $2$ by \eqref{eq:rank-one-ideals}. Hence $(A-B)Q=AQ=AQ+BQ$, and the result follows from \eqref{eq:sandwich-Q}.
		
		Otherwise, $AQ+BQ=Q$, and we show that $\Phi^Q_{A,B}$ is surjective. If $A-B$ is invertible, this follows from \eqref{eq:sandwich-Q}. So assume $\det(A-B)=0$, and let $v\in F^2$ be nonzero with $Av=Bv$. Since $AQ\neq BQ$, the lines $AF^2$ and $BF^2$ are distinct by \eqref{eq:rank-one-ideals}, so $Av=Bv\in AF^2\cap BF^2=0$. Hence $v\in\ker A\cap\ker B$, and since both kernels are lines, $\ker A=\ker B$. Now let $u$ and $w$ span $AF^2$ and $BF^2$, respectively. Then $\{u,w\}$ is a basis of $F^2$, and conjugating $A$ and $B$ by the matrix $P\in GL_2(F)$ with $Pu=e_1$ and $Pw=e_2$, we may assume $AF^2=Fe_1$ and $BF^2=Fe_2$, that is,
		$$
		A=\begin{pmatrix}a_1&a_2\\0&0\end{pmatrix},\qquad B=\begin{pmatrix}0&0\\b_1&b_2\end{pmatrix}.
		$$
		Since conjugation preserves the equality $\ker A=\ker B$, we have $(b_1,b_2)=\lambda(a_1,a_2)$ for some $\lambda\in F^\times$. Conjugating by $\operatorname{diag}(1,\lambda^{-1})$, we obtain
		$$
		A=\begin{pmatrix}c_1&c_2\\0&0\end{pmatrix},\qquad B=\begin{pmatrix}0&0\\c_1&c_2\end{pmatrix},\qquad (c_1,c_2)\neq(0,0).
		$$
		With respect to the basis $(E_{11},E_{12},E_{21},E_{22})$ of $Q$, the matrices of $\psi^Q_{I+E_{12}}$ and $\psi^Q_{I+E_{21}}$ are
		$$
		\begin{pmatrix}
			c_1&0&c_1+c_2&0\\
			0&c_1&0&c_1+c_2\\
			-c_1&0&-c_2&0\\
			-c_1&-c_1&-c_2&-c_2
		\end{pmatrix}
		\qquad\text{and}\qquad
		\begin{pmatrix}
			c_1+c_2&0&c_2&0\\
			0&c_1+c_2&0&c_2\\
			-c_1&-c_1&-c_2&-c_2\\
			0&-c_1&0&-c_2
		\end{pmatrix},
		$$
		whose determinants are $c_1^4$ and $c_2^4$, respectively. So one of these maps is bijective.
	\end{proof}

	\section{Concluding remarks}\label{sec:concluding}
	
	Lemma~\ref{lem:trace-zero-commutator} is the octonion analogue of a classical result: every matrix of trace zero in $M_n(F)$ is a commutator, as proved by Shoda \cite{Shoda} in characteristic zero and by Albert and Muckenhoupt \cite{MR83961} in general. For central simple associative algebras, with the reduced trace in place of the trace, Amitsur and Rowen proved the same for division algebras of degree $2$ and $3$ \cite{MR1286824}, while Jang \cite{Jang2026} recently announced a division algebra of degree $4$ with an element of reduced trace zero which is not a commutator. Since every $x\in\mathcal O$ satisfies $x^2-\operatorname{Tr}(x)x+N(x)1=0$, $\operatorname{Tr}$ plays the role of the reduced trace in $\mathcal O$. By Kleinfeld's theorem \cite{Kleinfeld1953}, a central simple alternative algebra which is not associative is an octonion algebra, so Lemma~\ref{lem:trace-zero-commutator} gives the following.
	
	\begin{corollary}
		\label{cor:alternative shoda}
		Let $R$ be a central simple alternative algebra over a field $F$ with $\operatorname{char}F\neq2$. If $R$ is not associative, then every element of $R$ of trace zero is a commutator in $R$.
	\end{corollary}
	
	In a central division algebra $D$, every element of reduced trace zero is a sum of two commutators \cite{MR1286824}, so the commutators span the subspace of elements of reduced trace zero. So the image of $\Phi_{1,1}(x,y)=xy-yx$ on $D$ is a vector subspace if and only if every element of reduced trace zero is a commutator. Hence, for the division algebra of \cite{Jang2026}, the image of $\Phi_{1,1}$ is not a vector subspace, in contrast with Lemma~\ref{lem: A=B image} and Theorem~\ref{thm:quaternion}.
	
	\begin{question}
		\label{q:CSA subspace}
		Let $R$ be a central simple algebra over $F$ of degree $n\geq3$. For which $A,B\in R$ is the image of the map $\Phi_{A,B}(x,y)=Axy-Byx$ a vector subspace of $R$? In particular, is it always a vector subspace when $R=M_n(F)$?
	\end{question}
	
	Our results also describe the image of $\Psi_{A,B}(x,y)=(xy)A-(yx)B$ on $\mathcal O$. Indeed, since the conjugation of $\mathcal O$ satisfies $\overline{ab}=\overline b\,\overline a$ and preserves the norm \cite[Chapter~1]{MR1763974}, we have $\overline{\Psi_{A,B}(x,y)}=\Phi_{\overline A,\overline B}(\overline y,\overline x)$, and so we obtain the following.
	
	\begin{corollary}
		\label{cor:right coefficients}
		For $A,B\in\mathcal O$, we have $\operatorname{Im}(\Psi_{A,B})=\overline{\operatorname{Im}(\Phi_{\overline A,\overline B})}$. In particular, $\operatorname{Im}(\Psi_{A,B})$ is a vector subspace of $\mathcal O$, $\Psi_{A,B}$ is surjective if $A$ or $B$ is invertible and $A\neq B$, $\operatorname{Im}(\Psi_{A,A})=\mathcal O_0A$, and $\operatorname{Im}(\Psi_{A,0})=\mathcal OA$.
	\end{corollary}
	
	On the other hand, the image of a generalized polynomial of degree $2$ need not be a vector subspace, as the following example shows.
	
	\begin{example}
		\label{ex:x(Ay)}
		Consider the map $f(x,y)=x(py)$ on the split octonion algebra $\mathcal O$. Since $N(f(x,y))=N(x)N(p)N(y)$ and $N(p)=0$, every element of $\operatorname{Im}(f)$ has norm zero. But $n=f(1,n)$ and $m=f(m,1)$ lie in $\operatorname{Im}(f)$, while $N(n+m)=\det(E_{12}+E_{21})=-1$. So $\operatorname{Im}(f)$ is not a vector subspace. The same argument works on $M_2(F)$, where the image of $xE_{11}y$ is the set of matrices of rank at most one.
	\end{example}
	
	By a \emph{generalized polynomial} on $\mathcal O$ we mean a linear combination of nonassociative monomials in the variables $x_1,\ldots,x_m$ and in elements of $\mathcal O$. The example above leads to the following question: for which multilinear generalized polynomials on an octonion algebra $\mathcal O$ is the image a vector subspace of $\mathcal O$?
	
	\section*{Acknowledgements}
	The first-named author is supported by S\~ao Paulo Research Foundation (FAPESP), grants 2025/02457-5 and 2024/14914-9, and by CNPq, grant 405779/2023-2. The second-named author is supported by a Postdoctoral Fellowship from the Indian Institute of Science Education and Research (IISER) Mohali.
	\printbibliography
	
\end{document}